\documentclass[a4paper,12pt]{article}

\usepackage{multicol}
\usepackage{amsmath}
\usepackage{a4wide}
\usepackage{amsfonts}
\usepackage{color}
\usepackage{amssymb,,amsmath}
\usepackage{graphicx}
\usepackage{amsfonts}
\newenvironment{proof}{\noindent\textbf{Proof:}}{\hspace{\stretch{1}}$\square$}

\newtheorem{theorem}{Theorem}[section]

\newtheorem{definition}{Definition}
\newtheorem{proposition}[theorem]{Proposition}
\newtheorem{lemma}[theorem]{Lemma}

\newcommand{\rF}{\mathcal{F}}
\newcommand{\rH}{\mathcal{H}}

\newcommand{\rM}{\mathcal{M}}

\newcommand{\rE}{\mathcal{E}}

\newcommand{\rS}{\mathcal{S}}
\newcommand{\rU}{\mathcal{U}}
\newcommand{\rD}{\mathcal{D}}

\newcommand{\rV}{\mathcal{V}}
\newcommand{\rW}{\mathcal{W}}
\newcommand{\rZ}{\mathcal{Z}}

\font\timesept=cmr7

\newcommand{\CC}{\mathbb{C}}

\newcommand{\EE}{\mathbb{E}}

\newcommand{\NN}{\mathbb{N}}

\newcommand{\PP}{\mathbb{P}}

\newcommand{\RR}{\mathbb{R}}

\def\Rp{\RR^+}

\def\O{\Omega}
\def\wt{\widetilde}
\def\indic{{\mathop{\rm 1\mkern-4mu l}}}
\def\qq{\qquad}

\def\wh{\widehat}
\def\wt{\widetilde}

\def\ol{\overline}

\def\indic{{\mathop{\rm 1\mkern-4mu l}}}

\def\Ran{\mathop{{\rm Ran}}\nolimits}

\def\lan{\langle}
\def\ran{\rangle}

\def\ecarte{\vphantom{\buildrel\bigtriangleup\over =}}

\def\Rp{{\RR^{+}\!}}

\def\ps#1#2{\left\langle #1\, ,\, #2\right\rangle}

\def\[{{\mathord{[\![}}}
\def\]{{\mathord{]\!]}}}

\def\norme#1{\left\| #1\right\|}
\def\normca#1{{\left\| #1\right\|}^2}
\def\ab#1{\left\vert #1\right\vert}

\def\L{\Lambda}
\def\O{\Omega}

\def\D{\Delta}
\def\a{\alpha}
\def\b{\beta}
\def\s{\sigma}
\def\m{\mu}

\def\d{\delta}
\def\e{\varepsilon}

\def\l{\lambda}

\def\o{\omega}
\def\wh{\widehat}
\def\td{\widetilde}
\def\ll{\left[}
\def\rr{\right]}

\def\br{\overline}

\def\to{\rightarrow}

\begin{document}

\title{Complex Obtuse Random Walks\\and their Continuous-Time Limits\footnote{Work supported by ANR project ``HAM-MARK" N${}^\circ$ ANR-09-BLAN-0098-01}}

\author{S. Attal, J. Deschamps and  C. Pellegrini}

\date{}

\maketitle

\begin{abstract}
We study a particular class of complex-valued random variables and their associated random walks: the complex obtuse random variables. They are the generalization to the complex case of the real-valued obtuse random variables which were introduced in \cite{A-E} in order to understand the structure of normal martingales in $\RR^n$.The extension to the complex case is mainly motivated by considerations from Quantum Statistical Mechanics, in particular for the seek of a characterization of those quantum baths acting as classical noises. The  extension of obtuse random variables to the complex case is far from obvious and hides very interesting algebraical structures. We show that complex obtuse random variables are characterized by a 3-tensor which admits certain symmetries which we show to be the exact 3-tensor analogue of the normal character for 2-tensors (i.e. matrices), that is, a necessary and sufficient condition for being diagonalizable in some orthonormal basis.  We discuss the passage to the continuous-time limit for these random walks and show that they converge in distribution to normal martingales in $\CC^N$. We show that the 3-tensor associated to these normal martingales encodes their behavior, in particular the diagonalization directions of the 3-tensor indicate the directions of the space where the martingale behaves like a diffusion and those where it behaves like a Poisson process. We finally prove the convergence, in the continuous-time limit, of the corresponding multiplication operators on the canonical Fock space, with an explicit expression in terms of the associated 3-tensor again.
\end{abstract}

\eject
\tableofcontents
\eject

\section{Introduction and Motivations}

\subsection{Generalities}
Real obtuse random variables are particular random variables which were defined in \cite{A-E} in order to understand the discrete-time analogue of normal martingales in $\RR^n$. They were shown to be deeply connected to the Predictable Representation Property and the Chaotic Representation Property for discrete-time martingales in $\RR^n$. They are kind of minimal, centered and normalized random variables in $\RR^n$ and they exhibit a very interesting underlying algebraical structure. This algebraical structure is carried by a certain natural 3-tensor associated to the random variable. This 3-tensor has exactly the necessary and sufficient symmetries for being diagonalizable in some orthonormal basis (that is, they satisfy the exact extension to 3-tensors of the condition for being real symmetric for 2-tensors). The corresponding orthonormal basis carries the behavior of the associated random walk and in particular of its continuous-time limit. It is shown in \cite{A-E} that, for the continuous-time limit, in the directions associated to the null eigenvalues, the limit process behaves like a diffusion process, while it behaves like a pure jump process in the directions associated to the non-null eigenvalues. In \cite{A-P2} it is concretely shown how the 3-tensors of the discrete-time obtuse random walks converge to the ones of normal martingales in $\RR^n$.

\smallskip
Since this initial work of \cite{A-E} was only motivated by Probability Theory and Stochastic Process considerations, there was no real need for an extension of this notion to the complex case. The need for such an extension has appeared naturally through considerations in Quantum Statistical Mechanics. More precisely, the underlying motivation is to characterize the onset of classical noises emerging from quantum baths, in the so-called model of Repeated Quantum Interactions. 

Repeated quantum interaction models are physical models, introduced and developed in \cite{A-P}, which consist in describing the Hamiltonian dynamics of a quantum system undergoing a sequence of interactions with an  environment made of a chain of identical systems. These
models were developed for they furnish toy models for quantum
dissipative systems, they are at the same time Hamiltonian and
Markovian, they spontaneously give rise to quantum stochastic
differential equations in the continuous time limit.  It has been proved in \cite{B-P} and \cite{BdP} that they constitute a good toy model for a quantum heat bath in some situations and that they can also give an account of the diffusive behavior of an electron in an electric field, when coupled to a heat bath. When adding to each step of the dynamics a measurement of the piece of the environment which has just interacted, we recover all the discrete-time quantum trajectories for quantum systems (\cite{Pel}, \cite{Pel2}, \cite{Pel3}). Physically, this model corresponds exactly to physical experiments such as the ones performed by S. Haroche \emph{et al} on the quantum trajectories of a photon in a cavity (\cite{Har}, \cite{Har2}, \cite{BBB}).

The discrete-time dynamics of these repeated interaction systems, as well as their continuous-time limits, give rise to time evolutions driven by quantum noises coming from the environment. These quantum noises emerging from the environment describe all the possible actions inside the environment (excitation, return to ground state, jumps in between two energy levels, ...). It is a remarkable fact that these quantum noises can also be combined together in order to give rise to classical noises. In discrete-time they give rise to any random walk, in continuous-time they give rise to many well-known stochastic processes among which are all the Levy processes. 

The point is that complex obtuse random variables and their continuous-time limits are the key for understanding what kind of classical noise is appearing at the end from the quantum bath. The 3-tensor helps to read directly from the Hamiltonian which kind of classical noise will be driving the evolution equation. This was our initial motivation for developing the complex theory of obtuse random variables and normal martingales in $\CC^N$.

\smallskip
Surprisingly, the extension of obtuse random variables, obtuse random walks and their continuous-time limits, to the complex case is far from obvious. The algebraical properties of the associated 3-tensors give rise to the same kind of behaviors as in the real case, but, as we shall see in this article, many aspects (such as the diagonalization theorem) become now really non-trivial.

\subsection{Examples}
Let us have here a more detailed discussion on these physical motivations underlying our study. These motivations do not appear anymore in the rest of the article which is devoted entirely to the probabilistic properties of complex obtuse random walks and their continuous-time limits, but we have felt that it could be of interest for the reader to have a clearer picture of the physical motivations which have brought us to consider the complex case extension of obtuse random walks. This part can be skipped by the reader, it has no influence whatsoever on the rest of the article, these physical applications are developed in detail in \cite{ADP}.

Let us illustrate this situation with concrete examples. Consider a quantum system, with state space $\rH_\rS$, in contact with a quantum bath of the form $\rH_\rE=\otimes_{\NN}\,\CC^2$, that is, a spin chain. Let us denote by $a^i_j$, $i,j=0,1$, the usual basis of elementary matrices on $\CC^2$ and by $a^i_j(n)$ the corresponding matrix but acting only on the $n$-th copy of $\CC^2$. The Hamiltonian for the interaction between $\rH_{\rS}$ and one copy of $\CC^2$ is of the typical form
$$
H_{\rm tot}=H_{\rS}\otimes I+I\otimes H_\rE+L\otimes a^0_1+L^*\otimes a^1_0\,.
$$
Assume that the interaction between these two parts lasts for a small amount of time $h$, then the associated unitary evolution operator is $U=e^{-i h H_{\rm tot}}$ which can be decomposed as $U=\sum_{i,j=0}^1 U^i_j\otimes a^i_j$ for some operators $U^i_j$ on $\rH_\rS$. 

The action of the environment (the spin chain) by acting repeatedly of the system $\rH_\rS$, spin by spin, each time for a time duration $h$, gives rises to a time evolution driven by a sequence of unitary operators $(V_n)$ which satisfies (cf \cite{A-P} for details)
$$
V_{n+1}=\sum_{i,j=0}^1 (U^i_j\otimes I)\, V_n\, (I\otimes a^i_j(n+1))\,.
$$
This describes a rather general discrete time evolution for a quantum system and the operators $a^i_j (n)$ here play the role of discrete time quantum noises, they describe all the possible innovations brought by the environment. 

In \cite {A-P} it is shown that if the total Hamiltonian $H_{\rm tot}$ is renormalized under the form
$$
H_{\rm tot}=H_{\rS}\otimes I+I\otimes H_\rE+\frac{1}{\sqrt{h}}\,\left(L\otimes a^0_1+L^*\otimes a^1_0\right)
$$
(this can be understood as follows: if the time duration of the interactions $h$ tends to 0, then then the interaction needs to be strengthen adequately if one wishes to obtain a non-trivial limit)
then the time evolution $(V_{nh})$ converges, when $h$ tends to 0, to a continuous-time unitary evolution $(V_t)$ satisfying an equation of the form
$$
dV_t=K\,V_t\, dt+L\, V_t\, da^*(t)-L^*\, V_t\, da(t)\,,
$$
which is a quantum stochastic differential equation driven by quantum noises $da(t)$ and $da^*(t)$ on some appropriate Fock space. In other words, we obtain a perturbation of a Schr\"odinger equation by some additional quantum noise terms. 

\smallskip
The point now is that in the special case where $L=L^*$ then the discrete-time evolution and its continuous-time limit are actually driven by classical noises, for some of the terms in the evolution equation factorize nicely and make appearing classical noises instead of quantum noises (the noises get grouped in order to furnish a family of commuting self-adjoint operators, that is, a classical stochastic process). Indeed, one can show (cf \cite{A-P2} and \cite{A-D}) that the discrete time evolution can be written under the form of classical random walk on the unitary group $\rU(\rH_{\rS})$:
$$
V_{n+1}=A\,V_n+B\,V_n\, X_{n+1}\,,
$$
where $(X_n)$ is a sequence of i.i.d. symmetric Bernoulli random variables. The continuous time limit, with the same renormalization as above, gives rise to a unitary evolution driven by a classical Brownian motion $(W_t)$:
$$
dV_t=\left(iH-\frac 12L^2\right)\,V_t\, dt+L\,V_t\, dW_t\,.
$$
The equation above is the typical one for the perturbation of a Schr\"odinger equation by means of a Brownian additional term, if one wants the evolution to keep unitary at all times.

This example is a very simple one and belongs to those which were already well-known (cf \cite{A-P2}); they involve real obtuse random variables and real normal martingales. 

\smallskip
The point now is that one can consider plenty of much more complicated examples of a choice for the Hamiltonian $H_{\rm tot}$ which would give rise to classical noises instead of quantum noises. Our motivation was to understand and characterize when such a situation appears and to read on the Hamiltonian which kind of noise is going to drive the dynamics. Let us illustrate this with a more complicated example. Assume now that the environment is made of a chain of 3-level quantum systems, that is, $\rH_\rE=\otimes_\NN\, \CC^3$. For the elementary interaction between the quantum system $\rH_\rS$ and one copy of $\CC^3$ we consider an Hamiltonian of the form
$$
H_{\rm tot}=H\otimes I+ A\otimes\left(
\begin{array}{ccc}
 0 & 5 & 0 \\
 -1+2 i & 0 & 4-2 i \\
 -2+4 i & 0 & 2i
\end{array}
\right)
+B\otimes\left(
\begin{array}{ccc}
 0 & 0 & 5 \\
 -2+4 i & 0 & 2+i\\
 1-2 i & -i & -1+2 i
\end{array}
\right)\,,
$$
which is self-adjoint under the condition $B=-(1/2)(A+(1+2i)A^*)$. 

In this case the quantum dynamics in discrete time happens to be driven by a classical noise too, but this does not appear obviously here! We will understand, with the tools developed in this article, that the resulting discrete time dynamics is of the form
$$
V_{n+1}=A_0\, V_n+A_1\, V_n\, X^1_{n+1}+A_2\, V_n\, X^2_{n+1}\,,
$$
where the random variables $(X^1_n\,,\, X^2_n)$ are i.i.d. in $\CC^2$ taking the values 
$$
v_1=\left(\begin{matrix}i\\1\end{matrix}\right)\ ,\qq 
v_2=\left(\begin{matrix}1\\-1+i\end{matrix}\right)\ ,\qq 
v_3=-\frac15\left(\begin{matrix}3+4i\\1+3i\end{matrix}\right)
$$
with probabilities $p_1=1/3$, $p_2=1/4$ and $p_3=5/{12}$ respectively.

\smallskip
Putting a $1/\sqrt{h}$ normalization factor in front of $A$ and $B$ and taking the limit $h$ goes to 0, we will show in this article, that this gives rise to a continuous time dynamics of the form
$$
dV_t=L_0\, V_t\, dt+ L_1\, V_t\, dZ^1_t+L_2\, V_t\, dZ^2_t\,,
$$
where $(Z^1\,,\, Z^2)$ is a normal martingale in $\CC^2$ given by
$$
\begin{cases}
Z^1_t&=\frac{2+i}{\sqrt{10}}\, W^1_t+\frac{i}{\sqrt{2}}\, W^2_t\\\\
Z^2_t&=\frac{-1+2i}{\sqrt{10}}\, W^1_t+\frac{1}{\sqrt{2}}\, W^2_t\,,
\end{cases}
$$
where $(W^1\,,\, W^2)$ is a 2-dimensional real Brownian motion.

\smallskip We will also see in this article how to produce any kind of example in $\CC^n$ which mixes Brownian parts and Poisson parts. 

The way these random walks and their characteristics are identified, the way the continuous-time limits and their characteristics are identified, are non-trivial and make use of all the tools we  develop along this article: associated doubly-symmetric 3-tensor, diagonalisation of the 3-tensor, probabilistic characteristics of the associated random walk, passage to the limit on the tensor, passage to the limit on the discrete-time martingale, identification of the law of the limit martingale, etc.

\subsection{Structure of the Article}
This article is structured as follows. In Section 2 we introduce the notions of \emph{obtuse systems}, \emph{obtuse random variables} and their associated 3-tensors. We show a kind of uniqueness result and we show that they generate all finitely supported random variables in $\CC^N$. 

In Section 3 we establish the important symmetries shared by the 3-tensors of obtuse random variables and we show one of our main results: these symmetries are the necessary and sufficient conditions for the 3-tensor to be diagonalizable in some orthonormal basis. We show how to recover the real case, which remarkably does not correspond to the real character of the 3-tensor but to a certain supplementary symmetry.

Section 4 is kind of preparatory to the continuous-time limit of complex obtuse random walks. In this section  we show an important connection between complex obtuse random variables and real ones. This connection will be the key for understanding the continuous-time limits. In Section 4 we gather all the results concerning this connection with the real obtuse random variables and its consequences. We recall basic results on real normal martingales and deduce the corresponding ones for the complex normal martingales. In particular we establish what is the complex extension of a \emph{structure equation}. We connect the behavior of the complex normal martingale to the diagonalization of its associated 3-tensor. 

In Section 5 we finally prove our continuous-time convergence theorems. First of all, via the convergence of the tensors, exploiting the results of \cite{Tav}, we prove a convergence in law for the processes. Secondly, in the framework of Fock space approximation by spin chains developed in \cite{Att}, we prove the convergence of the associated multiplication operators, with explicit formulas in terms of quantum noises.

We finally illustrate our results in Section 6 through 2 examples, showing up the different types of behavior.

\section{Complex Obtuse Random Variables}

\subsection{Obtuse Systems}

Let $N\in\NN^*$ be fixed. 
In $\CC^N$, an \emph{obtuse system} is a family of $N+1$ vectors $v_1,\ldots, v_{N+1}$ such that
$$
\ps{v_i}{v_j}=-1
$$
for all $i\not =j$. 
In that case we put
$$
\wh v_i=\left(
\begin{matrix}
 1 \\
 v_i   
\end{matrix}
\right)\in\CC^{N+1}\,,
$$
so that 
$$
\ps{\wh v_i}{\wh v_j}=0
$$
for all $i\not=j$. They then form an orthogonal basis of $\CC^{N+1}$.
We put
$$
p_i=\frac{1}{\normca{\wh v_i}}=\frac 1{1+\normca{v_i}}\,,
$$
for $i=1,\ldots N+1$. 

\begin{lemma}\label{L:proba}
We then have
\begin{equation}\label{E:proba}
\sum_{i=1}^{N+1} p_i=1
\end{equation}
and
\begin{equation}\label{E:centered}
\sum_{i=1}^{N+1} p_i\,v_i=0\,.
\end{equation}
\end{lemma}
\begin{proof}
We have, for all $j$,
$$
\ps{\sum_{i=1}^{N+1} p_i\,\wh v_i}{\wh v_j}=p_j\normca{\wh v_j}=1=\ps{\left(
\begin{matrix}
 1 \\
 0   
\end{matrix}
\right)}{\wh v_j}\,.
$$
As the $\wh v_j$'s form a basis, this means that
$$
\sum_{i=1}^{N+1} p_i\,\wh v_i=\left(
\begin{matrix}
 1 \\
 0   
\end{matrix}
\right)\,.
$$
This implies the two announced equalities.
\end{proof}

\begin{lemma}\label{L:normalized}
We also have
\begin{equation}\label{E:normalized}
\sum_{i=1}^{N+1} p_i\,\vert v_i\rangle\langle v_i\vert=I_{\CC^N}\,.
\end{equation}
\end{lemma}
\begin{proof}
As the vectors $(\sqrt{p_i}\,\wh v_i)_{i\in\{1,\ldots,N+1\}}$ form an orthonormal basis of $\CC^{N+1}$ we have
$$
I_{\CC^{N+1}}=\sum_{i=1}^{N+1} p_i\,\vert\wh v_i\rangle\langle\wh v_i\vert\,.
$$
Now put 
$$
u=
\left(
\begin{matrix}
1\\
  0  
\end{matrix}
\right)\qquad\mbox{and}\qquad
\wt v_i=
\left(
\begin{matrix}  0  \\v_i
\end{matrix}
\right)\,,
$$
for all $i=1, \ldots, N+1$. 
We get
\begin{align*}
I_{\CC^{N+1}}&=\sum_{i=1}^{N+1} p_i\,\vert u+\wt v_i\rangle\langle u+\wt v_i\vert\\
&=\sum_{i=1}^{N+1} p_i\,\vert u\rangle\langle u\vert+\sum_{i=1}^{N+1} p_i\,\vert u\rangle\langle \wt v_i\vert+\sum_{i=1}^{N+1} p_i\,\vert\wt v_i\rangle\langle u\vert+\sum_{i=1}^{N+1} p_i\,\vert\wt v_i\rangle\langle\wt v_i\vert\,.
\end{align*}
Using \eqref{E:proba} and \eqref{E:centered}, we get
$$
I_{\CC^{N+1}}=\vert u\rangle\langle u\vert+\sum_{i=1}^{N+1} p_i\,\vert \wt v_i\rangle\langle \wt v_i\vert\,.
$$
In particular we have
$$
\sum_{i=1}^{N+1} p_i\,\vert v_i\rangle\langle v_i\vert=I_{\CC^{N}}\,,
$$
that is, the announced equality.
\end{proof}

\bigskip
Let us consider an example that we shall follow along the article. On $\CC^2$, the 3 vectors
$$
v_1=\left(\begin{matrix}i\\1\end{matrix}\right)\ ,\qq 
v_2=\left(\begin{matrix}1\\-1+i\end{matrix}\right)\ ,\qq 
v_3=-\frac15\left(\begin{matrix}3+4i\\1+3i\end{matrix}\right)
$$
form an obtuse system of $\CC^2$. The associated $p_i$'s are then respectively
$$
p_1=\frac13\ ,\qq p_2=\frac14\ ,\qq p_3=\frac5{12}\,.
$$ 

\subsection{Obtuse Random Variables}

Consider a random variable $X$, with values in $\CC^N$, which can take only $N+1$ different non-null values $v_1,\ldots,v_{N+1}$ with strictly positive probability $p_1,\ldots, p_{N+1}$ respectively. 

We shall denote by $X^1, \ldots, X^N$ the coordinates of $X$ in $\CC^N$. 
We say that $X$ is \emph{centered} if its expectation is $0$, that is, if $\EE[X^i]=0$ for all $i$. We say that $X$ is \emph{normalized} if its covariance matrix is $I$, that is, if
$$
\mbox{cov}(X^i,X^j)=\EE[\ol{X^i}\, X^j]-\EE[\ol{X^i}]\,\EE[X^j]=\delta_{i,j}\,,
$$
for all $i,j=1,\ldots N$. 

We consider the canonical version of $X$, that is, we consider the probability space $(\O, \rF, \PP)$ where $\O=\left\{1, \dots, N+1\right\}$, $\rF$ is the full $\s$-algebra of $\O$, the probability measure $\PP$ is given by $\PP \left(\left\{i\right\}\right)= p_i$ and the random variable $X$ is given by $X(i)=v_i$, for all $i\in\O$. The coordinates of $v_i$ are denoted by $v_i^k$, for $k=1,\ldots, N$, so that $X^k(i)=v^k_i$. 

In the same way as above we put
$$
\wh v_i=\left(
\begin{matrix}
 1 \\
 v_i   
\end{matrix}
\right)\in\CC^{N+1}\,,
$$
for all $i=1,\ldots,N+1$.

We shall also consider the deterministic random variable $X^0$ on $(\O, \rF, P)$, which is always equal to $1$. For $i=0,\ldots, N$ let $\td{X}^i$ be the random variable defined by 
$$
\td{X}^i (j) =  \sqrt{p_j}\, X^i (j)
$$ 
for all $i=0,\ldots, N$ and all $j=1,\ldots, N+1$.

\begin{proposition}\label{P:ORV}
The following assertions are equivalent.

\smallskip\noindent
1) $X$ is centered and normalized.

\smallskip\noindent
2) The $(N+1)\times(N+1)$-matrix $\left(\wt{X}^i(j)\right)_{i,j}$ is a unitary matrix.

\smallskip\noindent
3) The $(N+1)\times(N+1)$-matrix $\left(\sqrt{p_i}\,\wh{v}_i^j\right)_{i,j}$ is a unitary  matrix.

\smallskip\noindent
4) The family $\{v_1, \dots, v_{N+1}\}$ is an obtuse system with 
$$
p_i= \frac{1}{1+\| v_i \|^2}\,,
$$
for all $i=1,\ldots, N+1$.
\end{proposition}
\begin{proof}\ 

\smallskip\noindent 
1) $\Rightarrow$ 2): Since the random variable $X$ is centered and normalized, each component $X^i$ has a zero mean and the scalar product between two components $X^i$, $X^j$ is given by the matrix $I$. Hence, for all $i$ in $\left\{1, \dots, N \right\}$, we get
\begin{equation}\label{centered}
\EE \ll X^i \rr=0 \quad\Longleftrightarrow \quad \sum_{k=1}^{N+1}  p_k\, v^i_k=0\,,
\end{equation}
and for all $i,j=1,\ldots N$,
\begin{equation}\label{normalized}
\EE[\ol{X^i}\, X^j ] = \delta_{i,j} \quad\Longleftrightarrow \quad \sum_{k=1}^{N+1}  p_k\, \ol{v^i_k}\, v^j_k = \delta_{i,j}\,.
\end{equation}
Now, using Eqs. (\ref{centered}) and (\ref{normalized}), we get, for all $i,j=1,\ldots,N$
\begin{align*}
\ps{\wt{X}^0}{\wt{X}^0}&=\sum_{k=1}^{N+1}p_k=1\,,\\
\ps{\td{X}^0}{\td{X}^i}& =  \sum_{k=1}^{N+1} \ol{\sqrt{p_k}}\, \sqrt{p_k}\, v_k^i =0\,,\\
\ps{\td{X}^i}{\td{X}^j} &=  \sum_{k=1}^{N+1} \ol{\sqrt{p_k}\, v_k^j}\, \sqrt{p_k} v_k^i = \d_{i,j}\,.
\end{align*}
The unitarity follows immediately.

\smallskip\noindent
2) $\Rightarrow$ 1): Conversely, if the matrix $\left(\td{X}^i(j)\right)_{i,j}$ is unitary, the scalar products of column vectors give the mean $0$ and the covariance $I$ for the random variable $X$.

\smallskip\noindent
2) $\Leftrightarrow$ 3): The matrix $\left(\sqrt{p_j}\, \wh{v}_i^j\right)_{i,j}$ is the transpose matrix of $\left(\td{X}^i(j)\right)_{i,j}$. Therefore, if one of these two matrices is unitary, its transpose matrix is unitary too.

\smallskip\noindent
3) $\Leftrightarrow$ 4): The matrix $\left(\sqrt{p_j}\, \wh{v}^i_j\right)_{i,j}$ is unitary if and only if 
$$
\ps{\sqrt{p_i}\, \wh{v}_i}{\sqrt{p_j}\, \wh{v}_j} = \d_{i,j}\,,
$$
 for all $i,j=1,\ldots,N+1$. On the other hand, the condition  
$\ps{\sqrt{p_i}\,\wh{v}_i}{\sqrt{p_i}\, \wh{v}_i} = 1$ is equivalent to $p_i\, (1 + \| v_i \|^2) =1$, 
whereas the condition
$\ps{\sqrt{p_i}\, \wh{v}_i}{\sqrt{p_j}\, \wh{v}_j} = 0$ is equivalent to $\sqrt{p_i}\, \sqrt{p_j}\,( 1 + \ps{v_i}{ v_j})=0$, that is, $ \ps{v_i}{v_j} =-1 \,.$ This gives the result.
\end{proof}

\begin{definition}
Random variables in $\CC^N$ which take only $N+1$ different values with strictly positive probability, which are centered and normalized, are called \emph{obtuse random variables in $\CC^N$}.
\end{definition}

\subsection{Generic Character of Obtuse Random Variables}

We shall here present several results which show the particular character of obtuse random variables. The idea is that somehow they generate all the finitely supported probability distributions on $\CC^N$.

First of all, we show that obtuse random variables on $\CC^N$ with a prescribed probability distribution $\{p_1,\ldots, p_{N+1}\}$ are essentially unique.

\begin{theorem}\label{YUX}
Let $X$ be an obtuse random variable of $\CC^N$ having $\{p_1,\ldots, p_{N+1}\}$ as associated probabilities. Then the following assertions are equivalent.

\smallskip\noindent
i) The random variable $Y$ is an obtuse random variable on $\CC^N$ with same probabilities $\{p_1,\ldots, p_{N+1}\}$.

\smallskip\noindent 
ii) There exists a unitary operator $U$ on $\CC^N$ such that $Y=UX$.  
\end{theorem}
\begin{proof}
One direction is obvious. If $Y=UX$, then $\EE[Y]=U\EE[X]=U0=0$ and 
$$
\EE[YY^*]=\EE[UXX^*U^*]=U\EE[XX^*]U^*=UIU^*=I\,.
$$
Hence $Y$ is a centered and normalized random variable on $\CC^N$, taking $N+1$ different values, hence it is an obtuse random variable. The probabilities associated to $Y$ are clearly the same as for $X$.

\smallskip
In the converse direction, let $v_1,\ldots, v_{N+1}$ be the possible values of $X$, associated to the probabilities $p_1,\ldots, p_{N+1}$ respectively. Let $w_1,\ldots, w_{N+1}$ be the ones associated to $Y$. In particular, the vectors 
$$
\wh{v_i}={\sqrt{p_i}}\left(\begin{matrix} 1\\v_i\end{matrix}\right)\,,\ \ i=1,\ldots ,N+1,
$$
form an orthonormal basis of $\CC^{N+1}$. The same holds with the $\wh{w_i}$'s. Hence there exists a unitary operator $V$ on $\CC^{N+1}$ such that $\wh{w_i}=V\wh{v_i}$, for all $i=1,\ldots, N+1$. 

We shall index the coordinates of $\CC^{N+1}$, from 0 to $N$ in order to be compatible with the embedding of $\CC^N$ that we have chosen. In particular we have
$$
V^0_0+\sum_{j=1}^N V^0_j v^j_i=1
$$
for all $i=1,\ldots, N+1$. 
This gives in particular
\begin{equation}\label{V0i}
\sum_{j=1}^N V^0_j(v^j_1-v^j_{i})=0
\end{equation}
for all $i\in\{2,\ldots, N+1\}$. 

As the $\wh{v_i}$'s are linearly independent then so are the $\sqrt{p_1}\wh{v_1}-\sqrt{p_i}\wh{v_i}$, for $i=2, \ldots, N+1$. Furthermore, we have 
$$
\sqrt{p_1}\wh{v_1}-\sqrt{p_i}\wh{v_i}=\left(\begin{matrix} 0\\v_1-v_i\end{matrix}\right)
$$
this means that the $v_1-v_i$'s, for $i=2,\ldots, N+1$, are linearly independent. 

As a consequence the unique solution of the system (\ref{V0i}) is $V^0_j=0$ for all $j=1,\ldots, N$. This implies $V^0_0=1$ obviously. 

The same kind of reasoning applied to the relation $\wh{v_i}=V^*\wh{w_i}$ shows that the column coefficients $V^j_0$, $j=1,\ldots, N$ are also all vanishing. Finally the operator $V$ is of the form
$$
V=\left(\begin{matrix}1&\langle 0\vert\\\vert 0\rangle&U\end{matrix}\right)\,,
$$
for some unitary operator $U$ on $\CC^N$. This gives the result.
\end{proof}

\smallskip
Having proved that uniqueness, we shall now prove that obtuse random variables generate all the other random variables (at least with finite support). First of all, a rather simple remark which shows that the choice of taking $N+1$ different values is the minimal one for centered and normalized random variables in $\CC^N$.

\begin{proposition}
Let $X$ be a centered and normalized random variable in $\CC^d$, taking $n$ different values $v_1,\ldots, v_n$, with probability $p_1,\ldots,p_n$ respectively. Then we must have 
$$
n\geq d+1\,.
$$
\end{proposition}
\begin{proof}
Let $X$ be centered and normalized in $\CC^d$, taking the values $v_1,\ldots, v_n$ with probabilities $p_1,\ldots, p_n$ and with $n\leq d$, that is, $n<d+1$. Put 
$$
\wt{X}^0=\left(\begin{matrix} \sqrt{p_1}\\\vdots\\\sqrt{p_n}\end{matrix}\right), \wt{X}^1=\left(\begin{matrix} \sqrt{p_1}v^1_1\\\vdots\\\sqrt{p_n}v^1_n\end{matrix}\right),\ldots, \wt{X}^d=\left(\begin{matrix} \sqrt{p_1}v^d_1\\\vdots\\\sqrt{p_n}v^d_n\end{matrix}\right)\,.
$$
They are $d+1$ vectors of $\CC^n$. We have, for all $i,j=1, \ldots, d$
\begin{align*}
\ps{\wt{X}^0}{\wt{X}^0}&=\sum_{i=1}^n p_i =1\\
\ps{\wt{X}^0}{\wt{X}^i}&=\sum_{k=1}^n p_k v^i_k=\EE[X^i]=0\\
\ps{\wt{X}^i}{\wt{X}^j}&=\sum_{k=1}^n p_k \ol{v^i_k}v^j_k=\EE[\ol{X^i}X^j]=\delta_{i,j}\,.
\end{align*}
The family of $d+1$ vectors $\wt{X}^0, \ldots, \wt{X}^d$ is orthonormal in $\CC^n$. This is impossible if $n<d+1$. 
\end{proof}

We can now state the theorem which shows how general, finitely supported, random variables on $\CC^d$ are generated by the obtuse ones. We concentrate only on centered and normalized random variables, for they obviously generate all the others, up to an affine transform of $\CC^d$. 

\begin{theorem}
Let  $n\geq d+1$ and let $X$ be a centered and normalized random variable in $\CC^d$, taking $n$ different values $v_1,\ldots, v_n$, with probabilities $p_1,\ldots,p_n$ respectively. 

If $Y$ is any obtuse random variable on $\CC^{n-1}$ associated to the probabilities $p_1,\ldots,p_n$, then there exists a partial isometry $A$ from $\CC^{n-1}$ to $\CC^d$ with $\Ran A =\CC^d$ and such  that 
$$
X=AY\,.
$$
\end{theorem}
\begin{proof}
Assume that the obtuse random variable $Y$ takes the values $w_1,\ldots, w_n$ in $\CC^{n-1}$.
We wish to find a $(n-1)\times d$ matrix $A$ such that
\begin{equation}\label{systemn}
A^i_1w^1_j+A^i_2w^2_j+\ldots+A^i_{n-1}w^{n-1}_j=v^i_j
\end{equation}
for all $i=1,\ldots, d$, all $j=1,\ldots, n$. 
In particular, for each fixed $i=1,\ldots, d$, we have the following subsystem of $n-1$ equations with $n-1$ variables $A^i_1, \ldots, A^i_{n-1}$:
\begin{equation}\label{system}
\begin{cases} \sum_{k=1}^{n-1} w^k_1A^i_k =v^i_1&\\
\sum_{k=1}^{n-1} w^k_2A^i_k =v^i_2&\\
\qq\vdots\\
\sum_{k=1}^{n-1} w^k_{n-1}A^i_k =v^i_{n-1}\,.&
\end{cases}
\end{equation}
The vectors 
$$
w_1=\left(\begin{matrix}w^1_1\\\vdots\\w^{n-1}_1\end{matrix}\right),\ldots, 
w_{n-1}=\left(\begin{matrix}w^1_{n-1}\\\vdots\\w^{n-1}_{n-1}\end{matrix}\right)
$$
are linearly independent. Thus so are the vectors
$$
\left(\begin{matrix}w^1_1\\\vdots\\w_{n-1}^1\end{matrix}\right),\ldots, 
\left(\begin{matrix}w_1^{n-1}\\\vdots\\w^{n-1}_{n-1}\end{matrix}\right)\,.
$$
Hence the system (\ref{system}) can be solved and furnishes the coefficients $A^i_k$, $k=1, \ldots, n-1$. 
We have to check that these coefficients are compatible with all the equations of (\ref{system}). Actually, the only  equation from (\ref{systemn}) that we have forgotten in (\ref{system}) is
$$
\sum_k A^i_kw^k_n=v^i_n\,.
$$
But this equation comes easily from the $n-1$ first equations if we sum them after multiplication by $p_j$:
$$
\sum_{j=1}^{n-1} \sum_k A^i_k p_jw^k_j=\sum_{j=1}^{n-1} p_j v^i_j\,.
$$
This gives, using $\EE[X^i]=\EE[Y^i]=0$,
$$
\sum_k A^i_k (-p_nw^k_n)=-p_nv^ i_n\,,
$$
which is the required relation.

\smallskip
We have proved the relation $X=AY$ with $A$ being a linear map from $\CC^{n-1}$ to $\CC^d$. 
The fact that $X$ is normalized can be written as $\EE[XX^*]=I_d$. But 
$$
\EE[XX^*]=\EE[AYY^*A^*]=A\EE[YY^*]A^*=A\,I_n\,A^*=AA^*\,.
$$
Hence $A$ must satisfy $AA^*=I_d$, which is exactly saying that $A$ is a partial isometry with range $\CC^d$. 
\end{proof}

\subsection{Associated 3-Tensors}

Obtuse random variables are naturally associated to some 3-tensors with particular symmetries. This is what we shall prove here.

In this article, a \textit{3-tensor on $\CC^n$} is an element of $(\CC^N)^*\otimes\CC^N\otimes\CC^N$, that is, a linear map from $\CC^N$ to $\CC^N\otimes \CC^N$. Coordinate-wise, it is represented by a collection of coefficients $(S^{ij}_k)_{i,j,k=1,  \dots, n}$. It acts  on $\CC^N$ as 
$$
(S(x))^{ij} = \sum_{k=1}^{n} S^{ij}_k x^k\,.
$$

We shall see below that obtuse random variables on $\CC^N$ have a naturally associated 3-tensor on $\CC^{N+1}$. Note that, because of our notation choice $X^0, X^1, \ldots, X^N$, the 3-tensor is indexed by $\{0,1,\ldots, N\}$ instead of $\{1,\ldots, N+1\}$. 

\begin{proposition}\label{P:STdiscret}
Let $X$ be an obtuse random variable in $\CC^N$. Then there exists a unique 3-tensor $S$ on $\CC^{N+1}$ such that
\begin{equation}\label{E:XX=TX}
X^i\, X^j=\sum_{k=0}^N S^{ij}_k\, X^k\,,
\end{equation}
for all $i,j=0,\ldots, N$. 
This $3$-tensor $S$ is given by
\begin{equation}\label{E:Tijk}
S^{ij}_k=\EE[{X^i}\, X^j\, \ol{X^k}]\,,
\end{equation}
for all $i,j,k=0,\ldots N$.

We also have the relation, for all $i,j=0,\ldots, N$
\begin{equation}\label{E:XbX}
\ol{X^i}\,X^j=\sum_{k=0}^N \ol{S^{ik}_j}\, X^k\,.
\end{equation}
\end{proposition}
\begin{proof}
As $X$ is an obtuse random variable, that is, a centered and normalized random variable in $\CC^N$ taking exactly $N+1$ different values, the random variables $\{X^0,X^1,\ldots,X^N\}$ are orthonormal in $L^2(\O,\rF,\PP)$, hence they form an orthonormal basis of $L^2(\O,\rF,\PP)$, for the latter space is $N+1$-dimensional. These random variables being bounded, the products ${X^i}\,X^j$ are still elements of $L^2(\O,\rF,\PP)$, hence they can be written, in a unique way, as linear combinations of the $X^k$'s. As a consequence, there exists a unique 3-tensor $S$ on $\CC^{N+1}$ such that
$$
{X^i}\, X^j=\sum_{k=0}^N S^{ij}_k\, X^k
$$
for all $i,j=0,\ldots,N$. In particular we have
$$
\EE[{X^i}\, X^j\, \ol{X^k}]=\sum_{l=0}^N S^{ij}_l\, \EE[X_l\,\ol{X_k}]=S^{ij}_k\,.
$$
This shows the identity (\ref{E:Tijk}).

Finally, we have, by the orthonormality of the $X^k$'s
$$
\ol{X^i}\, X^j=\sum_{k=0}^N \EE[\ol{X^i}\, X^j\,\ol{X^k}]\, X^k\,,
$$
that is,
$$
\ol{X^i}\, X^j=\sum_{k=0}^N \ol{S^{ik}_j}\, X^k\,,
$$
by (\ref{E:Tijk}). This gives the last identity.
\end{proof}

\bigskip
This 3-tensor $S$ has quite some symmetries, let us detail them.
\begin{proposition}\label{P:symmetries}
Let $S$ be the  3-tensor associated to an obtuse random variable $X$ on $\CC^N$. Then the 3-tensor $S$ satisfies the following relations, for all $i,j,k,l=0,\ldots N$
\begin{equation}\label{E:sym0}
S^{i0}_k=\d_{ik}\,,
\end{equation}
\begin{equation}\label{E:sym1}
S^{ij}_k \ \ \mbox{is symmetric in}\ \ (i,j)\,,
\end{equation}
\begin{equation}\label{E:sym2}
\sum_{m=0}^N {S^{im}_j}\,S^{kl}_m\ \ \mbox{is symmetric in}\ \ (i,k)\,,
\end{equation}
\begin{equation}\label{E:sym3}
\sum_{m=0}^N S^{im}_j\,\ol{S^{lm}_k}\ \ \mbox{is symmetric in}\ \ (i,k)\,.
\end{equation}
\end{proposition}
\begin{proof}\ 

\smallskip\noindent
The relation (\ref{E:sym0}) is immediate for
$$
S^{i0}_k=\EE[X^i\,\ol{X^k}]=\d_{ik}\,.
$$

\smallskip\noindent
Equation (\ref{E:sym1}) comes directly from Formula (\ref{E:Tijk}) which shows a clear symmetry in $(i,j)$.

\smallskip\noindent
By (\ref{E:XbX}) we have
$$
X^i\, \ol{X^j}=\sum_{m=0}^N S^{im}_j\, \ol{X^m}\,,
$$
whereas 
$$
X^k\,X^l=\sum_{n=0}^N S^{kl}_n\, X^n\,.
$$
Altogether this gives
$$
\EE\left[{X^i}\, \ol{X^j}\,{X^k}\, {X^l}\right]=\sum_{m=0}^NS^{im}_j\, {S^{kl}_m}\,.
$$
But the left hand side is clearly symmetric in $(i,k)$ and (\ref{E:sym2}) follows.

\smallskip\noindent
In order to prove (\ref{E:sym3}), we write, using (\ref{E:XbX})
$$
{X^i}\, \ol{X^j}=\sum_{m=0}^N {S^{im}_j}\, \ol{X^m}
$$
and
$$
\ol{X^l}X^k=\sum_{n=0}^N \ol{S^{lm}_k}\, X^n\,.
$$
Altogether we get
$$
\EE\left[{X^i}\, \ol{X^j}\,\ol{X^l}X^k\right]=\sum_{m=0}^N{S^{im}_j}\, \ol{S^{lm}_k}\,.
$$
But the left hand side is clearly symmetric in $(i,k)$ and (\ref{E:sym3}) is proved.
\end{proof}

\subsection{Representation of Multiplication Operators}\label{SS:mult}

Let $X$ be an obtuse random variable in $\CC^N$, with associated 3-tensor $S$ and let $(\O,\rF,\PP_S)$ be the canonical space of $X$. Note that we have added the dependency on $S$ for the probability measure $\PP_S$. The reason is that, when changing the obtuse random variable $X$ on $\CC^N$, the canonical space $\O$ and the canonical $\s$-field $\rF$ do not change, only the canonical measure $\PP$ does change.

We have seen that the space $L^2(\O,\rF,\PP_S)$ is a $N+1$-dimensional Hilbert space and that the family $\{{X^0}, {X^1},\ldots, {X^N}\}$ is an orthonormal basis of that space. Hence for every obtuse random variable $X$, with associated 3-tensor $S$, we have a natural unitary operator
$$
\begin{matrix} U_S&:&L^2(\O,\rF,\PP_S)&\longrightarrow&\CC^{N+1}\\
&&{X^i}&\longmapsto&e_i\,,
\end{matrix}
$$
where $\{e_0,\ldots, e_N\}$ is the canonical orthonormal basis of $\CC^{N+1}$. The operator $U_S$ is called the \emph{canonical isomorphism} associated to $X$. 

The interesting point with these isomorphisms $U_S$ is that they canonically transport all the obtuse random variables of $\CC^N$ onto a common canonical space. But the point is that the probabilistic informations concerning the random variable $X$ are not correctly transferred via this isomorphism: all the informations about the law, the independencies, ... are lost when identifying $X^i$ to $e_i$. The only way to recover the probabilistic informations about the $X^i$'s on $\CC^{N+1}$ is to consider the \emph{multiplication operator} by $X^i$, defined as follows.
On the space $L^2(\O,\rF,\PP_S)$, for each $i=0,\ldots, N$, we consider the multiplication operator
$$
\begin{matrix} \rM_{X^i}&:&L^2(\O,\rF,\PP_S)&\longrightarrow&L^2(\O,\rF,\PP_S)\\
&&{Y}&\longmapsto&X^i\,Y\,,
\end{matrix}
$$
These multiplication operators carry all the probabilistic informations on $X$, even through a unitary transform such as $U_S$, for we have, by the usual functional calculus for normal operators
\begin{align*}
\EE[f(X^1,\ldots, X^N)]&=\ps{X_0}{f(\rM_{X^1},\ldots, \rM_{X^N})\, X_0}_{L^2(\O,\rF,\PP_T)}\\
&=\ps{e_0}{U_S\,f(\rM_{X^1},\ldots, \rM_{X^N})\,U_S^*\, e_0}_{\CC^{N+1}}\\
&=\ps{e_0}{f(U_S\,\rM_{X^1}U_S^*,\ldots, U_S\,\rM_{X^N}\,U_S^*)\, e_0}_{\CC^{N+1}}\,.
\end{align*}
On the space $\CC^{N+1}$, with canonical basis $\{e_0,\ldots, e_N\}$ we consider the basic matrices $a^i_j$, for $i,j=0,\ldots, N$ defined by
$$
a^i_j\, e_k=\delta_{i,k}\, e_j\,.
$$
We shall see now that, when carried out on the same canonical space by $U_S$, the obtuse random variables of $\CC^N$ admit a simple and compact matrix representation in terms of their 3-tensor.

\begin{theorem}\label{T:multop}
Let $X$ be an obtuse random variable on $\CC^N$, with associated 3-tensor $S$ and canonical isomorphism $U_S$. Then we have, for all $i,j=0,\ldots, N$
\begin{equation}\label{UXU}
U_S\, \rM_{X^i}\, U_S^*=\sum_{j,k=0}^N S^{ij}_k\, a^j_k\,.
\end{equation}
for all $i=0,\ldots, N$. 

The operator of multiplication by $\ol{X^i}$ is given by
\begin{equation}\label{olXX}
U_S\,\rM_{\ol{X^i}}\, U_S^*=\sum_{j,k=0}^N \ol{S^{ik}_j}\, a^j_k\,.
\end{equation}
\end{theorem}
\begin{proof}
 We have, for any fixed $i\in\{0,\ldots,N\}$, for all $j=0,\ldots, N$
\begin{align*}
U_S\,\rM_{X^i}\, U_S^*\, e_j&=U_S\,\rM_{X^i}\, X^j\\
&=U_S\, X^i\,X^j\\
&=U_S\sum_{k=0}^N S^{ij}_k X^k\\
&=\sum_{k=0}^N S^{ij}_k e_k\,.
\end{align*}
Hence the operator $U_S\,\rM_{X^i}\, U_S^*$ has the same action on the orthonormal basis $\{e_0,\ldots,e_N\}$ as the operator
$$
\sum_{k=0}^N S^{ij}_k \, a^j_k\,.
$$
This proves the representation (\ref{UXU}).

The last identity is just an immediate translation of the relation (\ref{E:XbX}).
\end{proof}

\subsection{Back to the Example}

Let us illustrate the previous subsections with our example. To the obtuse system 
$$
v_1=\left(\begin{matrix}i\\1\end{matrix}\right)\ ,\qq 
v_2=\left(\begin{matrix}1\\-1+i\end{matrix}\right)\ ,\qq 
v_3=-\frac15\left(\begin{matrix}3+4i\\1+3i\end{matrix}\right)
$$
of $\CC^2$ is associated the random variable $X$ on $\CC^2$ which takes the values $v_1,v_2,v_3$ with probability $p_1=1/3$, $p_2=1/4$ and $p_3=5/{12}$ respectively. Then the 3-tensor $S$ associated to $X$ is directly computable.  We present $S$ as a collection of matrices $S^{j}=\left(S^{i\, j}_k\right)_{i,k}$\,, which are then the matrices of multiplication by $X^j$:
\begin{align*}
S^0&=\left(\begin{matrix}1&0&0\\\ecarte 0&1&0\\\ecarte0&0&1\end{matrix}\right)\,,
\qq S^1=\left(\begin{matrix}0&1&0\\\ecarte-\frac15(1-2i)&0&-\frac25(2+i)\\\ecarte
-\frac25(1-2i)&0&\frac15(2+i)\end{matrix}\right)\\
S^2&=\left(\begin{matrix}0&0&1\\\ecarte-\frac25(1-2i)&0&\frac15(2+i)\\\ecarte
\frac15(1-2i)&-i&-\frac15(1-2i)\end{matrix}\right)\,.
\end{align*}
These matrices are not symmetric (we shall see in Subsection \ref{SS:real} what the symmetry of the matrices $S^j$ corresponds to). We recognize the particular form of $S^0$, for it corresponds to $\rM_{X^0}=I$.

\section{Complex Doubly-Symmetric 3-Tensors}

We are going to leave for a moment the obtuse random variables and concentrate on the symmetries we have obtained above. The relation (\ref{E:sym0}) is really specific to obtuse random variables, we shall leave it for a moment. We concentrate on the relation (\ref{E:sym1}), (\ref{E:sym2}) and (\ref{E:sym3}) which have important consequences for the 3-tensor.

\subsection{The Main Diagonalization Theorem}

\begin{definition}
A 3-tensor $S$ on $\CC^{N+1}$ which satisfies (\ref{E:sym1}), (\ref{E:sym2}) and (\ref{E:sym3}) is called a \emph{complex doubly-symmetric 3-tensor on $\CC^{N+1}$}.
\end{definition}

The main result concerning complex doubly-symmetric 3-tensors in $\CC^{N+1}$ is that they are the exact generalization for 3-tensors of normal matrices for 2-tensors: they are exactly those 3-tensors which can be diagonalized in some orthonormal basis of $\CC^{N+1}$.

\begin{definition}
A 3-tensor $S$ on $\CC^{N+1}$ is said to be \emph{diagonalizable in some orthonormal basis} ${(a_m)}_{m=0}^N$ of $\CC^{N+1}$ if there exists complex numbers ${(\l_m)}_{m=0}^N$ such that
\begin{equation}\label{E:eigen1}
S=\sum_{m=0}^N \l_m\,a_m^*\otimes a_m\otimes a_m\,.
\end{equation}
In other words
\begin{equation}\label{E:eigen}
S(x)=\sum_{m=0}^N \l_m \,\ps{a_m}x\, a_m\otimes a_m
\end{equation}
for all $x\in\CC^{N+1}$. 
\end{definition}

Note that, as opposed to the case of 2-tensors (that is, matrices), the ``eigenvalues" $\l_m$ are not completely determined by the representation (\ref{E:eigen}). Indeed, if we put $\wt{a}_m=e^{i\theta_m}\, a_m$ for all $m$, then the $\wt{a}_m$'s still form an orthonormal basis of $\CC^{N+1}$ and we have
$$
S(x)=\sum_{m=1}^N \l_m\, e^{i\theta_m} \,\ps{\wt{a}_m}x\,  \wt{a}_m\otimes\wt{a}_m\,.
$$
Hence the $\l_m$'s are only determined up to a phase; only their modulus is determined by the representation (\ref{E:eigen}).

Actually, there are more natural objects that can be associated to diagonalizable 3-tensors; they are the \emph{orthogonal families in $\CC^N$}. Indeed, if $S$ is diagonalizable as above, for all $m$ such that $\l_m\not =0$ put 
$v_m={\l_m}\, a_m\,.$
The family $\{v_m\,;\ m=1,\ldots, K\}$ is then an orthogonal family in $\CC^{N+1}$ and we have
$$
S(v_m)=\vert v_m\rangle\langle v_m\vert
$$
for all $m$.
In terms of the $v_m$'s, the decomposition (\ref{E:eigen}) of $S$ becomes
\begin{equation}\label{E:eigen2}
S(x)=\sum_{m=1}^K \frac{1}{\normca{v_m}}\, \ps{v_m}x\,  v_m\otimes v_m\,.
\end{equation}
This is the form of diagonalization we shall retain for 3-tensors. Be aware that in the above representation the vectors are orthogonal, but not normalized anymore. Also note that they represent the eigenvectors of $S$ associated only to the non-vanishing eigenvalues of $S$.

We can now state the main theorem.

\begin{theorem}\label{T:diag}
A 3-tensor $S$ on $\CC^{N+1}$ is diagonalizable in some orthonormal basis if and only if it is doubly-symmetric. 

More precisely, the formulas
$$
\rV=\left\{v \in \CC^{N+1}\setminus\{0\}\,;\ S(v)= v\otimes v \right\}\,,
$$
and
$$
S(x)= \sum_{v \in \rV} \frac{1}{\normca v}\,  \ps{v}{x} \, v\otimes v\,,
$$
establish a bijection between the set of complex doubly-symmetric 3-tensors $S$ and the set of orthogonal systems $\rV$ in $\CC^{N+1}$.
\end{theorem}
\begin{proof}
Firste step: let $\rV=\{v_m\,;\ m=1,\ldots, K\}$ be an orthogonal familly in $\CC^{N+1}\setminus\{0\}$. Put
$$
S^{ij}_k = \sum_{m=1}^{K} \frac{1}{\normca{v_m}}\,{v_m^i}\, v_m^j\, \ol{v_m^k}\,,
$$
for all $i,j,k=0,\ldots, N$. We shall check that $S$ is a complex doubly-symmetric 3-tensor in $\CC^N$. 
The symmetry of $S^{ij}_k$ in $(i,j)$ is obvious from the definition. This gives (\ref{E:sym1}).

We have
\begin{align*}
\sum_{m=0}^N {S^{im}_j}\,{S^{kl}_m}&=\sum_{m=0}^N\sum_{n,p=1}^{K} \frac{1}{\normca{v_n}}\,\frac{1}{\normca{v_p}}\,{v_n^i}\, \ol{v_n^j}\, {v_n^m}\,\ol{v_p^m}\, v_p^k\, {v_p^l}\\
&=\sum_{n,p=1}^{K} \frac{1}{\normca{v_n}}\,\frac{1}{\normca{v_p}}\,{v_n^i}\, \ol{v_n^j}\, \ps{v_p}{v_n}\, v_p^k\, {v_p^l}\\
&=\sum_{n=1}^{K} \frac{1}{\normca{v_n}}\,{v_n^i}\, \ol{v_n^j}\, v_n^k\, {v_n^l}
\end{align*}
and the symmetry in $(i,k)$ is obvious. This gives (\ref{E:sym2}).

\smallskip\noindent
We have
\begin{align*}
\sum_{m=0}^N S^{im}_j\,\ol{S^{lm}_k}&=\sum_{m=0}^N\sum_{n,p=1}^{K}\frac{1}{\normca{v_n}}\,\frac{1}{\normca{v_p}}\, {v_n^i}\, v_n^m\, \ol{v_n^j}\,\ol{v_p^l}\, \ol{v_p^m}\, {v_p^k}\\
&=\sum_{n,p=1}^{K} \frac{1}{\normca{v_n}}\,\frac{1}{\normca{v_p}}\,{v_n^i}\, \ol{v_n^j}\, \ps{v_n}{v_m}\,{v_p^k}\, \ol{v_p^l}\\
&=\sum_{n=1}^{K} \frac{1}{\normca{v_n}}\,{v_n^i}\, \ol{v_n^j}\,v_n^k\, \ol{v_n^l}
\end{align*}
and the symmetry in $(i,k)$ is obvious. This gives (\ref{E:sym3}).

\smallskip\noindent 
We have proved that the formula
\begin{equation}\label{E:Tyyy}
S(x)= \sum_{v \in \rV}\frac 1{\normca v}\, \ps v x\,  v\otimes v
\end{equation}
defines a complex doubly-symmetric 3-tensor if $\rV$ is any family of (non-vanishing) orthogonal vectors. 

\smallskip
Second step: now given a complex doubly-symmetric 3-tensor $S$ of the form (\ref{E:Tyyy}), we shall prove that the set $\rV$ coincides with the set
$$
\wh\rV=\{v\in\CC^N\setminus\{0\}\,;\ S(v)= v\otimes v\}\,.
$$
Clearly, if $y\in \rV$ we have by (\ref{E:Tyyy})
$$
S(y)=y\otimes y\,.
$$
This proves that $\rV\subset \wh\rV$. Now, let $v\in \wh\rV$. On one side we have
$$
S(v)=v\otimes v\,,
$$
on the other side we have
$$
S(v)=\sum_{y\in \rV} \frac 1{\normca y}\, \ps y v\, y\otimes y\,.
$$
In particular, applying $\langle y\vert \in\rS^*$ to both sides, we get
$$
\ps yv\, v=\ps yv\, y
$$
and thus either $v$ is orthogonal to $y$ or $v=y$. This proves that $v$ is one of the elements $y$ of $\rV$, for it were orthogonal to all the $y\in\rS$ we would get $v\otimes v=S(v)=0$ and $v$ would be the null vector. 

We have proved that $\rV$ coincides with the set
$$
\{v\in\CC^N\setminus\{0\}\,;\ S(v)=\vert v\rangle\langle v\vert\}\,.
$$

Third step: now we shall prove that all complex doubly-symmetric 3-tensors $S$ on $\CC^{N+1}$ are diagonalizable in some orthonormal basis. 
The property (\ref{E:sym1}) indicates that the matrices
$$
S_k=(S^{ij}_k)_{i,j=1,\ldots, N}
$$
are symmetric. But, as they are complex-valued matrices, this does not imply any property of diagonalization. Rather we have the following theorem (\cite{H-H}).

\begin{theorem}[Takagi Factorization]\label{T:takagi}
Let M be a complex symmetric matrix, there exist a unitary $U$ matrix and a diagonal matrix $D$ such that
\begin{equation}
M=UDU^T=UD(\br U)^{-1}.
\end{equation}
\end{theorem}

\smallskip
Secondly, we shall need to simultaneously ``factorize'' the $S_k$'s as above. We shall make use of the following criteria (same reference).

\begin{theorem}[Simultaneous Takagi factorization]\label{thm:STF}
Let $\mathcal F = \left\{A_i\,;\ i \in \mathcal J \right\}$ be a family of complex symmetric matrices in $\mathcal M_n(\CC)$. Let $\mathcal G = \left\{\ol{A_i}\, A_j\,;\  i,j \in \mathcal J \right\}$. Then there exists a unitary matrix $U$ such that, for all $i$ in $\mathcal J$, the matrix $UA_i U^T$ is diagonal if and only if the family $\mathcal G$ is commuting.
\end{theorem}

This is the first part of Step three: proving that in our case the matrices $\ol{S_i}\, S_j$ commute. 
Using the 3 symmetry properties of $S$ we get
\begin{align*}
\left(\ol{S_i}\, S_j\, \ol{S_k}\,S_l\right)_{m,n}
&=\sum_{x,y,z=0}^N \ol{S^{mx}_i}\,S^{xy}_j\, \ol{S^{yz}_k}\, S^{zn}_l\\
&=\sum_{x,y,z=0}^N \ol{S^{mx}_y}\,S^{xi}_j\, \ol{S^{yz}_k}\, S^{zn}_l\\
&=\sum_{x,y,z=0}^N \ol{S^{zx}_y}\,S^{xi}_j\, \ol{S^{ym}_k}\, S^{zn}_l\\
&=\sum_{x,y,z=0}^N \ol{S^{zx}_n}\,S^{xi}_j\, \ol{S^{ym}_k}\, S^{zy}_l\\
&=\sum_{x,y,z=0}^N \ol{S^{zx}_i}\,S^{xn}_j\, \ol{S^{ym}_k}\, S^{zy}_l\\
&=\sum_{x,y,z=0}^N \ol{S^{my}_k}\, S^{yz}_l\,\ol{S^{zx}_i}\,S^{xn}_j \\
&=\left(\ol{S_k}\, S_l\, \ol{S_i}\,S_j\right)_{m,n}\,.
\end{align*}
This proves that $\ol{S_i}\, S_j\, \ol{S_k}\,S_l=\ol{S_k}\, S_l\, \ol{S_i}\, S_j$.
The family $\big\{\ol{S_i} S_j, i,j = 1, \cdots, N \big\}$ is commuting. Thus, by Theorem \ref{thm:STF}, the matrices $S_k$ can be simultaneously Takagi-factorized. There exists then a unitary matrix $U=(u^{ij})_{i,j = 0, \cdots, N}$ such that, for all $k$ in $\left\{0, \cdots, N \right\}$, 
\begin{equation}\label{diago}
S_k = U \,D_k\, \br U ^{-1}\,,
\end{equation}
where the matrix $D_k$ is a diagonal matrix, $D_k= diag(\l_k^1, \cdots, \l_k^N)$.
Thus, the coefficient $S^{ij}_k$ can be written as
$$
S^{ij}_k = \sum_{m=0}^N \l_k^m\, u^{im}\, u^{jm}\,.
$$
Let us denote by $a_m$ the $m$th column vector of $U$, that is, $a_m=({u^{lm}})_{l=0,\cdots, N}$. Moreover, we denote by $\l^m$ the vector of $\l^m_k$, for $k=0, \cdots, N$. Since the matrix $U$ is unitary, the vectors $a_m$ form an orthonormal basis of $\CC^{N+1}$. We have 
$$
S^{ij}_k = \sum_{m=0}^N {a^i_m}\, \l_k^m\, {a^k_m}\,.
$$
Our aim now is to prove that $\l_m$ is proportional to $\ol{a_m}$. To this end, we shall use the symmetry properties of $S$. From the simultaneous reduction (\ref{diago}), we get
$$
\ol{S_j}\, S_q= \ol U\, \ol{D_j}\, D_q\, {}^tU\,,
$$
where ${}^tU$ is the transpose matrix of $U$.
Thus, we have
$$
(\ol{S_j}\, S_q)_{i,r}=\sum_{m=0}^N \ol{S^{im}_j}\, S^{mr}_q=\sum_{m=0}^N \ol{a^i_m}\, \ol{\lambda^m_j}\, \lambda^m_q \,{a^r_m}\,.
$$
In particular we have, for all $p\in\{0,\ldots, N\}$
$$
\sum_{i,j,q,r=0}^N(\ol{S_j}\, S_q)_{i,r}\, {a^i_p}\, {\l_j^p}\, \ol{\l_q^p}\, \ol{a^r_p}= \sum_{m=0}^N \ps{{a_m}}{a_p}\, \ps{\l^m}{\l^p}\,\ps{\l^p}{\l^m}\, \ps{a_p}{{a_m}}=\norme{\l^p}^4\,.
$$
But applying the symmetry (\ref{E:sym3}) this is also equal to
\begin{align*}
\sum_{i,j,q,r=0}^N\sum_{m=0}^N \ol{a^q_m}\, \ol{\lambda^m_j}\, \lambda^m_i \,{a^r_m}\, {a^i_p}\, {\l_j^p}\, \ol{\l_q^p}\, \ol{a^r_p}&= \sum_{m=0}^N \ps{\ol{a_p}}{\l^m}\, \ps{\l^m}{\l^p}\,\ps{a_m}{\ol{\l^p}}\, \ps{a_p}{a_m}\\
&=\ab{\ps{\ol{a_p}}{\l^p}}^2\,\normca{\l^p}\,.
\end{align*}
This gives
$$
\ab{\ps{\ol{a_p}}{\l_p}}=\norme{\l_p}=\norme{\ol{a_p}}\,\norme{\l_p}\,.
$$
This is a case of equality in Cauchy-Schwartz inequality, hence there exists $\m_p\in\CC$ such that $\l^p=\m_p\, \ol{a_p}$, for all $p=0,\ldots,N$. This way, the 3-tensor $T$ can be written as
\begin{equation}
S^{ij}_k = \sum_{m=0}^N \mu_m \,{a^i_m}\, a^j_m\, \ol{a^k_m}\,.
\end{equation}
In other words
$$
S(x)=\sum_{m=0}^N \m_m\, \ps{a_m}x\, a_m\otimes a_m\,.
$$
We have obtained the orthonormal diagonalization of $S$. The proof is complete.
\end{proof}

\subsection{Back to Obtuse Random Variables}

The theorem above is a general diagonalization theorem for 3-tensors. For the moment it does not take into account the relation (\ref{E:sym0}). When we make it enter into the game, we see the obtuse systems appearing.

\begin{theorem}
Let $S$ be a doubly-symmetric 3-tensor on $\CC^{N+1}$ satisfying also the relation
$$
S^{i0}_k=\delta_{ik}
$$
for all $i,k=0,\ldots, N$.
Then the orthogonal system $\rV$ such that
\begin{equation}\label{E:SxObtuse}
S(x)=\sum_{v\in \rV} \frac 1{\norme{v}^2} \ps{v}{x}\,v\otimes v
\end{equation}
is made of exactly $N+1$ vectors $v_1, \ldots, v_{N+1}$, all of them satisfying $v^0_i=1$. In particular the family of $N+1$ vectors of $\CC^{N}$, obtained by restricting the $v_i$'s to their $N$ last coordinates, forms an obtuse system in $\CC^{N}$.
\end{theorem}

\begin{proof}
First assume that $\rV=\{v_1,\ldots, v_K\}$. By hypothesis, we have
$$
S^{ij}_k=\sum_{m=1}^K \frac 1{\normca{v_m}}\, {v^i_m}\, v^j_m\, \ol{v^k_m}\,,
$$
for all $i,j,k=0,\ldots, N$.
With hypothesis (\ref{E:sym0}) we have in particular
$$
S^{i0}_k=\sum_{m=1}^K \frac 1{\normca{v_m}}\, {v^i_m}\, v^0_m\, \ol{v^k_m}=\delta_{ik}
$$
for all $i,k=0,\ldots, N$.

Consider the orthonormal family of $\CC^{N+1}$ made of the vectors $e_m=v_m/\norme{v_m}$. We have obtained above the relation
$$
\sum_{m=0}^K {v^0_m}\, \vert e_m\rangle\langle e_m\vert =I
$$
as matrices acting on $\CC^{N+1}$. The above is thus a spectral decomposition of the identity matrix, this implies that the $e_m$'s are exactly $N+1$ vectors and that all the ${v^0_m}$ are equal to 1.

This proves the first part of the theorem. The last part concerning obtuse systems is now obvious and was already noticed when we have introduced obtuse systems.
\end{proof}

\bigskip
In particular we have proved the following theorem. 

\begin{theorem}
The set of doubly-symmetric 3-tensors $S$ on $\CC^{N+1}$ which satisfy also the relation
$$
S^{i0}_k=\delta_{ik}
$$
for all $i,k=0,\ldots, N$, is in bijection with the set of obtuse random variables $X$ on $\CC^N$. 
The bijection is described by the following, with the convention $X^0=\indic$:

\smallskip\noindent
-- The random variable $X$ is the only random variable satisfying
$$
{X^i}X^j=\sum_{k=0}^N S^{ij}_k X^k\,,
$$
for all $i,j=1,\ldots, N$.

\smallskip\noindent
-- The 3-tensor $S$ is obtained by
$$
S^{ij}_k=\EE[{X^i}\, X^j\, \ol{X^k}]\,,
$$
for all $i,j,k=0,\ldots,N$. 

\smallskip In particular the different possible values taken by $X$ in $\CC^N$ coincide with the vectors $w_n\in\CC^N$, made of the last $N$ coordinates of the eigenvectors $v_n$ associated to $S$ in the representation (\ref{E:SxObtuse}). The associated probabilities are then $p_n=1/(1+\normca{w_n})=1/\normca{v_n}$. 
\end{theorem}

\subsection{Recovering the Real Case}\label{SS:real}

In \cite{A-E} have been introduced the notions of \emph{real} obtuse random variables and their associated \emph{real} doubly-symmetric 3-tensors. In the same way they obtained certain symmetries on the tensor which corresponded exactly to the condition for being diagonalizable in some real orthonormal basis. Note that in \cite{A-E} the situation for the diagonalization theorem was much easier, for the symmetries associated to the 3-tensor came down to simultaneous diagonalization of commuting symmetric real matrices.

The question we want to answer here is: How do we recover the real case from the complex case? By this we mean: On what condition a complex doubly-symmetric 3-tensor correspond to a real one, that is, corresponds to real-valued random variables?
 Surprisingly enough, the answer is not: When the coefficients $S^{ij}_k$ are all real! Let us see that with a counter-example.
 
Let us consider the one dimensional random variable $X$ which takes values $i$, $-i$ with probability $1/2$. As usual denote by $X^0$ the constant random variable equal to 1 and by $X^1$ the random variable $X$. We have the relations
\begin{align*}
{X^0}X^0 &= X^0\,\\
{X^0}X^1 &= X^1\,\\
{X^1}X^0 &= X^1\,\\
{X^1}X^1 &= -X^0\,
\end{align*}
which give us the following matrices for the associated 3-tensor $S$:
\begin{align*}
S^0&=(S^{i0}_k)_{i,k}=\rM_{X^0}=\left( \begin{array}{cc}
1&0\\
0&1
\end{array} \right)\\
 S^{1}&=(S^{i1}_k)_{i,k}=\rM_{X^1}=\left( \begin{array}{cc}
0&-1\\
1&0
\end{array} \right)\,.
\end{align*}
They are real-valued matrices, but they are associated to a complex (non real) random variable.
 
\smallskip
In fact, the major difference between a complex (non real) doubly-symmetric 3-tensor and a real doubly-symmetric 3-tensor is the commutation property of indices $i$ and $k$ in the coefficients $S^{ij}_k$. Let us make this more precise. 

\begin{definition}
A doubly symmetric 3-tensor $S$ on $\CC^N$ which also satisfies (\ref{E:sym0}) is said to be \emph{real} if the associated obtuse random variable $X$ on $\CC^N$ is only real-valued.
\end{definition}

\begin{proposition}\label{P:criterionR}
Let $S$ be a complex doubly symmetric tensor on $\CC^N$ which satisfies (\ref{E:sym0}). The following assertions are equivalent.

\smallskip\noindent
1) For all $i,j,k$  we have
$$
S^{ij}_k=S^{kj}_i\,.
$$

\smallskip\noindent
2) The 3-tensor $S$ is real.
\end{proposition}

\begin{proof}
The commutation relation implies that
$$
\mathbb{E}[X^i\,{X^j}\,\ol{X^k}]= \mathbb{E}[{X^k}\,X^j\,\ol{X^i}], 
$$
for all $i,j,k=0, 1,\ldots, N$. Since $\{X_j\,;\ j=0,1,\ldots,N\}$ is an orthonormal basis of the canonical space $L^2(\O,\rF,\PP_S)$, we get 
$$
X^i\,\ol{X^k}=\ol{X^i}\,X^k\,,\ \ a.s
$$
for all $i,k$. Then $\ol{X^k}\,X^i$ is almost surely real for all $i,k$. Considering the case $k=0$ implies that $X^i$ is almost surely real and the result follows.
\end{proof}

\medskip
In the counter-example above, one can check that $S^{01}_1= 1$ and $S^{11}_0=-1$.  The commutation condition is not satisfied.

\subsection{From Complex to Real Obtuse Random Variables}

The aim of this subsection is to prove that every complex obtuse random variable is obtained by a unitary transform of $\CC^N$ applied to some real obtuse random variable. This will be obtained in several steps, here is the first one.

\begin{proposition}\ 

\smallskip\noindent
1) Let $X$ be an obtuse random variable in $\CC^N$ and let $X^1, \ldots, X^N$ be its coordinate random variables.
If $Y^1,\ldots Y^n$ are {\bf real} random variables of the form $Y=U\,X$, with $U$ being a unitary operator of $\CC^N$, then  $Y$ is a real obtuse random variable of $\RR^N$. 

\smallskip\noindent
2) Conversely, if $Y$ is a real obtuse random variable on $\RR^N$ and if $X=UY$ with $U$ being a unitary operator on $\CC^N$, then $X$ is an obtuse random variable in $\CC^N$.
\end{proposition}
\begin{proof}
This is essentially the same argument as in Theorem \ref{YUX}, at least for the second property. For the first property one has to write that, if $Y$ is real-valued and $Y=UX$ then 
$$
Y^t=Y^*
$$
so that 
$$
\EE[Y\,Y^t]=\EE[Y\, Y^*]=U\,\EE[X\,X^*]\,U^*=I
$$
in the same way as in the proof of Theorem \ref{YUX}. Hence $Y$ is a centered and normalized real random variable in $\RR^N$, taking $N+1$ different values, hence it is a real obtuse random variable of $\RR^N$, as is proved in \cite{A-E} in a theorem similar to Proposition \ref{P:ORV}.
\end{proof}

\bigskip
Now recall the following classical result.

\begin{proposition}\label{P:schur}
Let $v_1,\ldots, v_{N}$ be any linearly free family of $N$ vectors of $\CC^N$. Then there exists a unitary operator $U$ on $\CC^N$ such that the vectors $w_i=Uv_i$ are of the form
$$
w_1=\left(\begin{matrix} z_1^1\\0\\\vdots\\0\\0\end{matrix}\right),\ \ 
w_2=\left(\begin{matrix} z_2^1\\z_2^2\\\vdots\\0\\0\end{matrix}\right),\ \ \ldots,
\ \ w_{N-1}=\left(\begin{matrix} z_{N-1}^1\\z_{N-1}^2\\\vdots\\z_{N-1}^{n-1}\\0\end{matrix}\right),\ \ 
w_{N}=\left(\begin{matrix} z_{N}^1\\z_{N}^2\\\vdots\\z_{N}^{N-1}\\z_{N}^N\end{matrix}\right)\,.
$$
\end{proposition}
\begin{proof}
It is clear that with a well chosen unitary operator $U$ one can map $v_1$ onto $\CC e_1$. The family $w_1,\ldots,w_N$ of images by $U$ of $v_1, \ldots, v_N$ is free. Furthermore, if we put 
$$
w_i=\left(\begin{matrix} a_i\\y_i\end{matrix}\right)
$$
with $a_i\in\CC$ and $y_i\in\CC^{N-1}$, then we claim that the family $\{y_2,\ldots,y_N\}$ is free in $\CC^{N-1}$. Indeed, we must have $a_1\not=0$ and if
$$
\sum_{i=2}^{N} \l_i\,y_i=0
$$
then
$$
\sum_{i=2}^N\l_i\, w_i-\frac{\sum_{i=2}^N \l_i\,a_i}{a_1}w_1=0
$$
and all the $\l_i$ vanish.

Once this has been noticed, we consider the unitary operator $V$ on $\CC^{N-1}$ which maps $y_2$ onto $(1,0,\ldots,0)$ and the unitary operator $U'$ on $\CC^N$ given by
$$
U'=\left(\begin{matrix}1&0&\ldots&0\\0&&&\\0&&V&\\0&&&&\end{matrix}\right)\,.
$$
We repeat the procedure until all the coordinates are exhausted.
\end{proof}

\bigskip 
Now, here is an independence property specific shared by the obtuse systems.
\begin{proposition}\label{free}
Every strict sub-family of an obtuse family is linearly free.
\end{proposition}
\begin{proof}
Let $\{v_1,\ldots,v_{N+1}\}$ be an obtuse family of $\CC^N$. 
Let us show that $\{v_1,\ldots,v_N\}$ is free, which would be enough for our claim. If we had
$$
v_N=\sum_{i=1}^{N-1} \l_i \,v_i
$$
then, taking the scalar product with $v_N$ we would get 
$$
\normca{v_N}=\sum_{i=1}^{N-1} -\l_i\,,
$$
whereas, taking the scalar product with $v_{N+1}$ would give
$$
-1=\sum_{i=1}^{N-1} -\l_i\,.
$$
This would imply $\normca{v_N}=-1$, which is impossible.
\end{proof}

\bigskip
Finally, using Proposition \ref{P:schur} we make an important step towards the main result.
\begin{proposition}
Let $\{w_1,\ldots,w_{N+1}\}$ be an obtuse system of $\CC^N$ such that
$$
w_1=\left(\begin{matrix} z_1^1\\0\\\vdots\\0\\0\end{matrix}\right),
w_2=\left(\begin{matrix} z_2^1\\z_2^2\\\vdots\\0\\0\end{matrix}\right),\ldots,
w_{N-1}=\left(\begin{matrix} z_{N-1}^1\\z_{N-1}^2\\\vdots\\z_{N-1}^{N-1}\\0\end{matrix}\right).
$$
Then there exist $\phi_1, \ldots, \phi_N$,  modulus 1 complex numbers, such that for every $i=1,\ldots N+1$ we have
$$
w_i=\left(\begin{matrix}\phi_1\,a_i^1\\\vdots\\\phi_N\,a_i^N\end{matrix}\right)
$$
with the $a_i^j$'s being reals.
\end{proposition}
\begin{proof}
First note that the $z^i_i$'s, $i=1,\ldots ,N$, cannot vanish, for otherwise, the family $\{w_1,\ldots, w_N\}$ would not be linearly free, contradicting Proposition \ref{free}.

Secondly, the scalar product conditions $\ps{w_1}{w_j}=-1$, for $j=2,\ldots, N+1$, imply 
$$
z^1_2=\ldots=z^1_{N+1}=-\left(\ol{z^1_1}\right)^{-1}\,.
$$
In particular, all the $z^1_i$'s, $i=1,\ldots, N+1$, have the same argument.

With the conditions $\ps{w_2}{w_j}=-1$, for $j=3,\ldots, N+1$, we get
$$
\ol{z^2_2}z^2_3=\ldots=\ol{z^2_2}z^2_{N+1}=-1-\ab{z^1_2}^2.
$$
Hence all the $z^2_i$'s are equal for $i=3, \ldots, n+1$ and all $z^2_i$'s have same argument ($i=2, \ldots, N+1$).

One easily obtains the result in the same way, line by line.
\end{proof}

\bigskip
Altogether we have proved the following theorem.

\begin{theorem}\label{T:CtoR}
For every obtuse family  $v_1,\ldots, v_{N+1}$ of $\CC^N$ there exists a unitary operator $U$ of $\CC^N$ such that the vectors 
$w_i=Uv_i$ all have real coordinates. 
This family of vectors of $\RR^N$ form a real obtuse system of $\RR^N$, with same probabilities as the initial family $v$.

In other words, every complex obtuse random variable $X$ in $\CC^N$ is of the form $X=UY$ for some unitary operator $U$ on $\CC^N$ and some real obtuse random variable $Y$ on $\RR^N$.
\end{theorem}

\subsection{Unitary Transforms of Obtuse Random Variables}

As every complex obtuse random variable $X$ can be obtained as $UY$ for some unitary operator $U$ and some real obtuse random variable $Y$, we shall concentrate for a while on the unitary transformations of obtuse random variables and their consequences on the associated 3-tensors, on the multiplication operators, etc.

As a first step, let us see how is transformed the associated 3-tensor under a unitary map of the random variable.

\begin{lemma}\label{L:StoT}
Let $X$ and $Y$ be two obtuse random variables on $\CC^N$ such that there exist a unitary operator $U$ on $\CC^N$ satisfying
$$
X=UY\,.
$$
We extend $\CC^N$ to $\CC^{N+1}$ by adding some $e_0$ vector to the orthonormal basis; we extend $U$ to a unitary operator on $\CC^{N+1}$ by imposing $Ue_0=e_0$. Let $(u_{ij})_{i,j=0,\ldots, N}$ be the coefficients of $U$ on $\CC^{N+1}$. 

If $S$ and $T$ are the 3-tensors of $X$ and $Y$ respectively, we then have
\begin{equation}\label{E:SUoT}
S^{ij}_k=\sum_{m,n,p=0}^N u_{im}\,u_{jn}\,\ol{u_{kp}}\,T^{mn}_p\,,
\end{equation}
for all $i,j,k=0,\ldots,N$.

Conversely, the tensor $T$ can be deduced from the tensor $S$ by
$$
T^{ij}_k=\sum_{m,n,p=0}^N \ol{u_{mi}}\,\ol{u_{nj}}\,u_{pk}\,S^{mn}_p\,.
$$
\end{lemma}
\begin{proof}
With the extension of $U$ to $\CC^{N+1}$ and the coordinates $X^0=Y^0=\indic$ associated to $X$ and $Y$ as previously, we have
$$
X^i=\sum_{m=0}^N u_{im} Y^m
$$
for all $i=0,\ldots, N$. 

The 3-tensor $S$ is given by
\begin{align*}
S^{ij}_k&=\EE[X^i\,X^j\,\ol{X^k}]\\
&=\sum_{m,n,p=0}^N u_{im}\, u_{jn}\, \ol{u_{kp}}\, \EE[Y^m\, Y^n\, \ol{Y^p}]\\
&=\sum_{m,n,p=0}^N u_{im}\, u_{jn}\, \ol{u_{kp}}\, T^{mn}_p\,.
\end{align*}

The converse formula is obvious, replacing $U$ by $U^*$. 
\end{proof}

\begin{definition}
In the following if two 3-tensors $S$ and $T$ are connected by a formula of the form (\ref{E:SUoT}) we shall denote it by
$$
S=U\circ T\,.
$$
\end{definition}

\smallskip
Let us see now what are the consequences on the representation of multiplication operators, such as given by Theorem \ref{T:multop}. First of all notice that if $X=UY$ then the underlying probability measures $\PP_S$ and $\PP_T$ of their canonical spaces are the same; the unitary transform does not change the probabilities, only the values. Hence $X$ and $Y$ are defined on the same probability space $(\O,\rF,\PP)$. Though, the canonical isomorphisms $U_T$ and $U_S$ are different, they differ by a change of basis actually.

\begin{proposition}\label{P:change_opmult}
Under the conditions and notations above, we have
$$
U_T\,\rM_{Y^i}\, U_T^*=\sum_{j,k=0}^N T^{ij}_k\, a^j_k
$$
and 
\begin{align}
U_S\, \rM_{X^i}\, U_S^*&=\sum_{j,k=0}^N\sum_{m,n,p=0}^N u_{im}\,u_{jn}\,\ol{u_{kp}}\, T^{mn}_p\, a^j_k\label{rMUY1}\\
&=\sum_{j,k=0}^N S^{ij}_k\, a^j_k\,.\label{rMUY2}
\end{align}
\end{proposition}
\begin{proof}
We have
\begin{align*}
U_S\, \rM_{X^i}\, U_{S^*}&=\sum_{m=0}^N u_{im}\, U_S\, \rM_{Y^m}\, U_{S^*}\\
&=\sum_{m=0}^N u_{im}\, U_S\,U_{T^*}\,U_T\, \rM_{Y^m}\, U_{T^*}\,U_T\,U_{S^*}\\
&=\sum_{m=0}^N\sum_{n,p=0}^N u_{im}\, T^{mn}_p\, U_S\,U_{T^*}\, a^n_p\,U_T\,U_{S^*}\,.
\end{align*}
An explicit formula for the operator $U_S\,U_{T^*}\, a^n_p\,U_T\,U_{S^*}$ is obtained easily by acting on the basis:
\begin{align*}
U_S\,U_{T^*}\, a^n_p\,U_T\,U_{S^*}\, e_j&=U_S\,U_{T^*}\, a^n_p\,U_T\,X^j\\
&=\sum_{k=0}^N u_{jk}\,U_S\,U_{T^*}\, a^n_p\,U_T\,Y^k\\
&=\sum_{k=0}^N u_{jk}\,U_S\,U_{T^*}\, a^n_p\,e_k\\
&= u_{jn}\,U_S\,U_{T^*}\,e_p\\
&= u_{jn}\,U_S\,Y^p\\
&= \sum_{k=0}^N u_{jn}\,\ol{u_{kp}}\, U_S\,X^k\\
&= \sum_{k=0}^N u_{jn}\,\ol{u_{kp}}\, e_k\,.
\end{align*}
This proves that
$$
U_S\,U_{T^*}\, a^n_p\,U_T\,U_{S^*}= \sum_{j,k=0}^N u_{jn}\,\ol{u_{kp}}\, a^j_k\,.
$$
Injecting this in the previous identity, we get
$$
U_S\, \rM_{X^i}\, U_{S^*}=\sum_{m=0}^N\sum_{n,p=0}^N \sum_{j,k=0}^Nu_{im}\, u_{jn}\,\ol{u_{kp}}\,T^{mn}_p \, a^j_k\,.
$$
That is, we get (\ref{rMUY1}) and (\ref{rMUY2}) immediately.
\end{proof}

\bigskip
The point is that this unitary operator has not been yet obtained very constructively. The following theorem gives it a little more explicitly, from the associated 3-tensor.

\begin{theorem}\label{T:UUt}
Let $X$ be a normal martingale in $\CC^N$ with associated 3-tensor $S$. Then the matrix $S_0=(S^{ij}_0)_{i,j=0,\ldots, N}$ is symmetric and unitary, it can be decomposed as $V\, V^t$ for some unitary matrix $V$ of $\CC^N$. For any such unitary operator $V$ the random variable $R=V^*\, X$ is a real obtuse random variable of $\CC^N$. 
\end{theorem}
\begin{proof}
By definition we have $S^{ij}_0=\EE[X^i\, X^j]$ and hence is symmetric in $(i,j)$. Now let us check it is a unitary matrix. We have
\begin{align*}
\sum_{m=0}^N S^{im}_0\,\ol{S^{jm}}_0&=\sum_{m=0}^N \EE[X^i\, X^m]\, \EE[\ol{X^m}\, \ol{X^j}]\\
&=\sum_{m=0}^N \ps{\ol{X^i}}{X^m}\,\ps{X^m}{\ol{X^j}}&\mbox{for the scalar product of }L^2(\O,\rF,\PP)\\
&=\ps{\ol{X^i}}{\ol{X^j}}&\mbox{for the }X^m\mbox{'s form an o.n.b. of }L^2(\O,\rF,\PP)\\
&=\EE[X^i\, \ol{X^j}]=\d_{ij}\,.
\end{align*}
We have proved the unitarity.

By Takagi Theorem \ref{T:takagi}, this matrix $S_0$ can be decomposed as $U\,D\,U^t$ for some unitary $U$ and some diagonal matrix $D$. But as $S_0$ is unitary we have
$$
I=S^*\, S=\ol{U}\, \ol{D} \,U^*\, U\, D\, U^t\,,
$$
hence
$$
\ab{D}^2=U^t\, \ol{U}=I
$$
and the matrix $D$ is unitary too. In particular its entries are complex numbers of modulus 1. Let $L$ be the diagonal matrix whose entries are the square root of the entries of $D$, they are also of modulus 1, so that $L\,\ol{L}=I$

Put $V=U\,L$, then 
$$
V\, V^t=U\, L \,{L}\, U^t=U\, D\, U^t=S_0\,,
$$
but also 
$$
V\, V^*=U\, L\,\ol{L}\, U^*=U\,U^*=I\,.
$$
We have proved the announced decomposition of $S_0$. 

\smallskip
We now check the last assertion. Let $v_{ij}$ be the coefficients of $V$. Define the 3-tensor $R=V^*\circ S$, that is,
$$
R^{ij}_k=\sum_{m,n,p=0}^N \ol{v_{mi}}\,\ol{v_{nj}}\, v_{pk}\, S^{mn}_p\,.
$$
Computing $S^{ij}_k=\EE[X^i\, X^j\, \ol{X^k}]$ in another way, we get
\begin{align*}
S^{ij}_k&=\EE[X^i\, X^j\, \ol{X^k}]\\
&=\sum_{m=0}^N \ol{S^{km}_j}\, \EE[X^i\, X^m]\\
&=\sum_{m=0}^N S^{im}_0\, \ol{S^{km}_j}\,.
\end{align*}
Injecting this relation in the expression of $R^{ij}_k$ above, we get
\begin{align*}
R^{ij}_k&=\sum_{m,n,p,\a,\b=0}^N \ol{v_{mi}}\,\ol{v_{nj}}\, v_{pk}\, v_{m\b}\, v_{\a\b}\,\ol{S^{p\a}_n}\\
&=\sum_{n,p,\a,\b=0}^N \d_{i\b}\,\ol{v_{nj}}\, v_{pk}\, v_{\a\b}\,\ol{S^{p\a}_n}\\
&=\sum_{n,p,\a=0}^N \,\ol{v_{nj}}\, v_{pk}\, v_{\a i}\,\ol{S^{p\a}_n}\,.
\end{align*}
But the above expression is clearly symmetric in $(i,j)$, for $S^{p\a}_n$ is symmetric in $(p,\a)$. By Proposition \ref{P:criterionR} this means that the 3-tensor $R$ is real. The theorem is proved.
\end{proof}

\section{Complex Normal Martingales}

The aim of next section is to give explicit results concerning the continuous-time limit of random walks made of sums of obtuse random variables. The continuous time limits will give rise to particular martingales on $\CC^{N}$. In the real case, the limiting martingales are well-understood, they are the so-called \emph{normal martingales} of $\RR^N$ satisfying a \emph{structure equation} (cf \cite{A-E}). In the real case the stochastic behavior of these martingales is intimately related to a certain doubly symmetric 3-tensor and to its diagonalization. To make it short, the directions corresponding to the null eigenvalues  of the limiting 3-tensor are those where the limit process behaves like a Brownian motion; the other directions (non-vanishing eigenvalues) correspond to a Poisson process behavior. This was developed in details in \cite{A-E} and we shall recall their main results below. 

\smallskip
In the next section of this article we wish to obtain two types of time-continuous results: 

\smallskip\noindent
-- a limit in distribution for the processes, for which we would like to rely on the results of \cite{Tav} where is proved that the convergence of the 3-tensors associated to the discrete time obtuse random walks implies the convergence in law of the processes;

\smallskip\noindent
-- a limit theorem for the multiplication operators, for which we would like to rely on the approximation procedure developed in \cite{Att}, where is constructed an approximation of the Fock space by means of spin chains and where is proved the convergence of the basic operators $a^i_j(n)$ to the increments of quantum noises.

\smallskip
When considering the complex case we had two choices: either develop a complex theory of normal martingales and structure equations, extend all the results of \cite{A-E}, of \cite{Tav} and of \cite{Att} to the complex case and prove the limit theorems we wished to obtain; or find a way to connect the complex obtuse random walks to the real ones and rely on the results of the real case, in order to derive the corresponding one for the complex case. We have chosen the second scenario, for we have indeed the same connection between the complex obtuse random variables and the complex ones as we have obtained in the discrete time case. In this section we shall present, complex normal martingales and their structure equations,  the connection between the complex and the real case, together with their consequences. Only in next section we shall apply these results in order to derive the  continuous-time limit theorems.

\subsection{A Reminder of Normal Martingales in $\RR^N$}

We now recall the main results of \cite{A-E} concerning the behavior of normal martingales in $\RR^N$ and their associated 3-tensor.

\smallskip
\begin{definition}
A martingale $X=(X^1,\ldots,X^N)$ with values in $\RR^N$ is a \emph{normal martingale} if  $X_0=0$ a.s. and if its angle brackets satisfy
\begin{equation}\label{E:angle}
\lan X^i\,,\, X^j\ran_t=\d_{ij}\, t\,,
\end{equation}
for all $t\in\Rp$ and all $i,j=1,\ldots, N$.  

This is equivalent to saying that the process $(X^i_tX^j_t-\d_{ij}\,t)_{t\in\Rp}$ is a martingale, or else that the process $([X^i\,,\,X^j]_t-\d_{ij}\, t)_{t\in\Rp}$ is a martingale (where $[\,\cdot\,,\,\cdot\,]$ here denotes the square bracket). 
\end{definition}

\begin{definition}
A normal martingale $X$ in $\RR^N$ is said to satisfy a \emph{structure equation} if there exists a family $\{\Phi^{ij}_k\,;\ i,\,j,\,k =1,\ldots,N\}$ of predictable processes such that
\begin{equation}\label{ES}
[X^i\,,\,X^j]_t=\d_{ij}\, t+\sum_{k=1}^N\int_0^t \Phi^{ij}_k(s)\, dX^k_s\,.
\end{equation}
\end{definition}

Note that if $X$ has the predictable representation property (i.e. every square integrable martingale is a stochastic integral with respect to $X$) and if $X$ is $L^4$, then $X$ satisfies a structure equation, for (\ref{ES}) is just the integral representation of the square integrable martingale $([X^i\,,\,X^j]_t-\d_{ij}\, t)_{t\in\Rp}$.

\smallskip
The following theorem is proved in \cite{A-E}. It establishes the fundamental link between the 3-tensors  $\Phi(s)$ associated to the martingale $X$ and the behavior of $X$. 

\begin{theorem}\label{T:AE}
Let $X$ be a normal martingale in $\RR^N$ satisfying the structure equation
$$
[X^i\,,\,X^j]_t=\d_{ij}\, t+\sum_{k=1}^N\int_0^t \Phi^{ij}_k(s)\, dX^k_s\,,
$$
for all $i,j$. Then for almost all $(s,\o)$  the quantities $\Phi(s,\o)$ are all valued in doubly symmetric 3-tensors in $\RR^N$. 

If one denotes by $\rV_s(\o)$ the orthogonal family associated to the non-vanishing eigenvalues of $\Phi(s,\o)$ and by $\Pi_s(\o)$ the orthogonal projector onto $\rV_s(\o)^\perp$, that is, on the null-egeinvalue subspace of $\Phi(s,\o)$, then the continuous part of $X$ is
$$
X^c_t=\int_0^t \Pi_s(dX_s)\,,
$$
the jumps of $X$ only happen at totally inaccessible times and they satisfy
$$
\D X_t(\o)\in\rV_t(\o)\,.
$$ 
\end{theorem}

\smallskip
The case we are concerned with is a simple case where the process $\Phi$ is actually constant. In that case, things can be made much more explicit, as is proved in \cite{A-E} again.

\begin{theorem}\label{T:AE2}
Let $\Phi$ be a doubly-symmetric 3-tensor in $\RR^N$, with associated orthogonal family $\rV$. Let $W$ be a Brownian motion with values in the space $\rV^\perp$. For every $v\in\rV$, let $N^v$ be a Poisson process with intensity $\norme{v}^{-2}$. We suppose $W$ and all the $N^v$ to be independent processes. 

Then the martingale
\begin{equation}\label{E:explicit}
X_t=W_t+\sum_{v\in\rV} \left(N_t^v-\frac{1}{\normca{v}}\, t\right)\, v
\end{equation}
satisfies the structure equation
\begin{equation}\label{E:ESconst}
[X^i\,,\,X^j]_t-\d_{ij}\, t=\sum_{k=1}^N\Phi^{ij}_k(t)\, X^k_t\,,
\end{equation}
Conversely, any solution of (\ref{E:ESconst}) has the same law as $X$.

The martingale $X$ solution of (\ref{E:ESconst}) possesses the chaotic representation property.
\end{theorem}

\subsection{Normal Martingales in $\CC^N$}

We consider a martingale $X$, defined on its canonical space $(\O,\rF,\PP)$, with values in $\CC^N$, satisfying the following properties (similar to the corresponding definition in $\RR^N$):

\smallskip\noindent
-- the angle bracket $\langle \ol{X^i}\,,\,X^j\rangle_t$ is equal to $\d_{ij}\, t$,

\smallskip\noindent
-- the martingale $X$ has the Predictable Representation Property.

\smallskip
To these conditions we add the following simplifying condition: 

\smallskip\noindent
-- the functions $t\mapsto \EE\left[[X^i\,,\,X^j]_t\right]$ are absolutely continuous with respect to the Lebesgue measure.

Applying all these conditions, we know that, for all $i,j,k=1,\ldots, N$,  there exist predictable processes $(\L^{ij}_t)_{t\in\Rp}$, $(S^{ij}_k(t))_{t\in\Rp}$ and $(T^{ij}_k(t))_{t\in\Rp}$, such that
\begin{align}
[X^i\,,\,X^j]_t&=\int_0^t \L^{ij}_s\, ds+\sum_{k=1}^N \int_0^t S^{ij}_k(s)\, dX^k_s\,,\label{ESCgen}\\
[\ol{X^i}\,,\,X^j]_t&=\d_{ij}\, t+\sum_{k=1}^N \int_0^t T^{ij}_k(s)\, dX^k_s\,.\label{ESCgen2}
\end{align}

Before proving the main properties and symmetries of the coefficients $(\L^{ij}_t)_{t\in\Rp}$, $(S^{ij}_k(t))_{t\in\Rp}$ and $(T^{ij}_k(t))_{t\in\Rp}$, we shall prove a uniqueness result.

\begin{lemma}\label{L:unique}
If $A$ and $B^k$, $k=1,\ldots, N$ are predictable processes on $(\O,\rF,\PP)$ such that
\begin{equation}\label{E:unique}
\int_0^t A_s\, ds+\sum_{k=1}^N \int_0^t B^k_s\, dX^k_s=0
\end{equation}
for all $t\in\Rp$, then $A_t=B^k_t=0$ almost surely, for almost all $t\in\Rp$, for all $k=1,\ldots, N$.
\end{lemma}
\begin{proof}
Let us write 
$$
Y_t=\int_0^t A_s\, ds+\sum_{k=1}^N \int_0^t B^k_s\, dX^k_s\,.
$$
Then
\begin{align*}
\ab{Y_t}^2&=\ol{Y_t}\,Y_t=\int_0^t \ol{Y_s}\, dY_s+\int_0^tY_s\, d\ol{Y_s}+[\ol{Y}\,,\,Y]_t\\
&=\int_0^t \ol{Y_s}\, dY_s+\int_0^tY_s\, d\ol{Y_s}+\sum_{k,l=1}^N\int_0^t\ol{B^k_s}\,B^l_s\, d[\ol{X^k}\,,\,X^l]_s\,.
\end{align*}
In particular,
$$
\EE[\ab{Y_t}^2]=\int_0^s \EE\left[\ol{Y_s}\, A_s+Y_s\,\ol{A_s}\right]\, ds+\sum_{k=1}^N\int_0^t \EE\left[\ab{B^k_s}^2\right]\, ds\,.
$$
If $Y$ is the null process then $B^k_s=0$ almost surely, for a.a. $s$ and for all $k$. This means that 
$$
\int_0^t A_s\, ds=0 
$$
for all $t$ and thus $A$ vanishes too.
\end{proof}

\medskip
We now detail the symmetry properties of $S$, $T$ and $\L$, together with some intertwining relations between $S$ and $\L$. 

\begin{theorem}\label{T:reduc_ESC}\ 

\smallskip\noindent
1) The processes $(S(t))_{t\in\Rp}$ and $(T(t))_{t\in\Rp}$ are connected by the relation
\begin{equation}\label{E:TS}
T^{ij}_k(s)=\ol{S^{ik}_j(s)}\,,
\end{equation}
almost surely, for a.a. $s\in\Rp$ and for all $i,j,k=1,\ldots, N$. 

\smallskip\noindent
2) The process $(S(t))_{t\in\Rp}$ takes its values in the set of doubly-symmetric 3 tensors of $\CC^N$.

\smallskip\noindent
3) The process $(\L_t)_{t\in\Rp}$ takes its values in the set of complex symmetric matrices.

\smallskip\noindent
4) We have the relation
\begin{equation}\label{E:ML}
\sum_{m=1}^N S^{ij}_m(s)\,\L^{mk}_s=\sum_{m=1}^N S^{kj}_m(s)\, \L^{mi}_s\,,
\end{equation}
almost surely, for a.a. $s\in\Rp$ and for all $i,j,k=1,\ldots, N$.

\smallskip\noindent
5) We have the relation
\begin{equation}\label{E:SbL}
\sum_{m=1}^N \ol{S^{km}_j(s)}\, \L^{im}_s=S^{ij}_k(s)\,,
\end{equation}
almost surely, for a.a. $s\in\Rp$ and for all $i,j,k=1,\ldots, N$.

\smallskip\noindent
6) The process $(\ol{X^i_t})_{t\in\Rp}$ has the predictable representation
$$
\ol{X^i_t}=\sum_{m=1}^N \int_0^t \ol{\L^{im}_s}\, dX^m_s\,.
$$

\smallskip\noindent
7) The matrix $\L$ is unitary.
\end{theorem}
\begin{proof}
The proof is a rather simple adaptation of the arguments used in Proposition \ref{P:STdiscret} and Proposition \ref{P:symmetries}. First of all, the symmetry $[X^i\,,\,X^j]_t=[X^j\,,\,X^i]_t$ gives
$$
\int_0^t \L^{ij}_s\, ds+\sum_{k=1}^N\int_0^t S^{ij}_k(s)\, dX^k_s=
\int_0^t \L^{ji}_s\, ds+\sum_{k=1}^N\int_0^t S^{ji}_k(s)\, dX^k_s
$$
for all $t$. By the uniqueness Lemma \ref{L:unique} this gives the symmetry of the matrices $\L_s$ and the first symmetry relation (\ref{E:sym1}) for the 3-tensors $S(s)$.

\smallskip
Computing $[[X^i\,,\,X^j]\,,\, \ol{X^k}]_t$ we get
\begin{align*}
[[X^i\,,\,X^j]\,,\, \ol{X^k}]_t&=\sum_{m=1}^N\int_0^t S^{ij}_m(s)\, d[X^m\,,\, \ol{X^k}]_s\\
&=\int_0^t S^{ij}_k(s)\, ds+\sum_{m,n=1}^N\int_0^t S^{ij}_m(s)\, T^{km}_n(s)\, dX^n_s\,.
\end{align*}
But this triple bracket is also equal to
\begin{align*}
[X^i\,,\,[X^j\,,\, \ol{X^k}]]_t&=\sum_{m=1}^N\int_0^t \ol{T^{jk}_m(s)}\, d[X^i\,,\,\ol{X^m}]_s\\
&=\int_0^t \ol{T^{jk}_i(s)}\, ds+\sum_{m,n=1}^N\int_0^t \ol{T^{jk}_m(s)}\, T^{mi}_n(s)\, dX^n_s\,.
\end{align*}
Again, by the uniqueness lemma, and the symmetry (\ref{E:sym1}), we get the relation (\ref{E:TS}).

\smallskip
Now, in the same way as in the proof of Proposition \ref{P:symmetries}, we compute $[[X^i\,,\, \ol{X^j}]\,,\,[X^k\,,\,X^l]]_t$ in two ways, using the symmetry in $(i,k)$ of that quadruple bracket:
\begin{align*}
[[X^i\,,\, \ol{X^j}]\,,\,[X^k\,,\,X^l]]_t&=\int_0^t \sum_{m,n=1}^N\int_0^t \ol{T^{ij}_m(s)}\,S^{kl}_n(s)\, d[\ol{X^m}\,,\,X^n]_s\\
[[X^k\,,\, \ol{X^j}]\,,\,[X^i\,,\,X^l]]_t&=\int_0^t \sum_{m,n=1}^N\int_0^t \ol{T^{kj}_m(s)}\,S^{il}_n(s)\, d[\ol{X^m}\,,\,X^n]_s\,.
\end{align*}
By uniqueness again, the time integral part gives the relation
$$
\sum_{m=1}^N \ol{T^{ij}_m(s)}\,S^{kl}_m(s)=\sum_{m=1}^N\ol{T^{kj}_m(s)}\,S^{il}_m(s)\,,
$$
that is,
$$
\sum_{m=1}^N {S^{im}_j(s)}\,S^{kl}_m(s)=\sum_{m=1}^N{S^{km}_j(s)}\,S^{il}_m(s)\,,
$$
which is the relation (\ref{E:sym2}). 

\smallskip
The relation (\ref{E:sym3}) is obtained exactly in the same way, from the symmetry of $[[X^i\,,\,\ol{X^j}]\,,\,[\ol{X^l}\,,\,X^k]]_t$ in $(i,k)$. We have proved that the 3-tensors $S(s)$ are doubly-symmetric.

\smallskip
Computing $[X^i\,,\, [X^j\,,\, X^k]]_t$ in two different ways we get
\begin{align*}
[X^i\,,\, [X^j\,,\, {X^k}]]_t&=\sum_{m=1}^N\int_0^t S^{jk}_m(s)\, d[X^i\,,\,X^m]_s\\
&=\sum_{m=1}^N\int_0^t S^{jk}_m(s)\, \L^{im}_s\, ds+\sum_{mn=1}^N\int_0^t S^{jk}_m(s)\, S^{im}_n(s)\, dX^n_s\,,
\end{align*}
on one hand, and
\begin{align*}
[[X^i\,,\, X^j]\,,\, {X^k}]_t&=\sum_{m=1}^N\int_0^t S^{ij}_m(s)\, d[X^m\,,\,X^k]_s\\
&=\sum_{m=1}^N\int_0^t S^{ij}_m(s)\, \L^{mk}_s\, ds+\sum_{mn=1}^N\int_0^t S^{ij}_m(s)\, S^{mk}_n(s)\, dX^n_s\,,
\end{align*}
on the other hand. Identifying the time integrals, we get the relation (\ref{E:ML}).

\smallskip
Finally, computing $[X^i\,,\, [X^j\,,\, \ol{X^k}]]_t$ in two different ways we get
\begin{align*}
[X^i\,,\, [X^j\,,\, \ol{X^k}]]_t&=\sum_{m=1}^N\int_0^t \ol{S^{jm}_k(s)}\, d[X^i\,,\,X^m]_s\\
&=\sum_{m=1}^N\int_0^t \ol{S^{jm}_k(s)}\, \L^{im}_s\, ds+\sum_{m,n=1}^N\int_0^t S^{jm}_k(s)\, S^{im}_n(s)\, dX^n_s\,,
\end{align*}
on one hand, and
\begin{align*}
[[X^i\,,\, X^j]\,,\, \ol{X^k}]_t&=\sum_{m=1}^N\int_0^t S^{ij}_m(s)\, d[X^m\,,\,\ol{X^k}]_s\\
&=\sum_{m=1}^N\int_0^t S^{ij}_m(s)\, \d_{mk}\, ds+\sum_{mn=1}^N\int_0^t S^{ij}_m(s)\, \ol{S^{mn}_k(s)}\, dX^n_s\,,
\end{align*}
on the other hand. Identifying the time integrals, we get the relation (\ref{E:SbL}).

\smallskip
Let us prove the result 6). The process $(\ol{X^i_t})_{t\in\Rp}$ is obviously a square integrable martingale and its expectation is 0. Hence it admits a predictable representation of the form
$$
\ol{X^i_t}=\sum_{k=1}^N\int_0^t H^{ik}_s\, dX^k_s
$$
for some predictable processes $H^{ik}$. We write
\begin{align*}
[\ol{X^i}\,,\,\ol{X^j}]_t&=\ol{[X^i\,,\, X^j]_t}\\
&=\int_0^t \ol{\L^{ij}_s}\, ds+\sum_{k=1}^N \int_0^t \ol{S^{ij}_k(s)}\, d\ol{X^k_s}\\
&=\int_0^t \ol{\L^{ij}_s}\, ds+\sum_{k,l=1}^N \int_0^t \ol{S^{ij}_k(s)}\, H^{kl}_s\,d{X^l_s}
\end{align*}
and
\begin{align*}
[\ol{X^i}\,,\,\ol{X^j}]_t&\sum_{k=1}^N\int_0^t H^{ik}_s\, d[X^k\,,\,\ol{X^j}]_s\\
&=\sum_{k=1}^N\int_0^t H^{ik}_s\, \d_{kj}\, ds+ \sum_{k,l=1}^N H^{ik}_s\, \ol{S^{jl}_k(s)}\,.
\end{align*}
The unicity lemma gives the relation $H^{ij}_s=\ol{\L^{ij}_s}$, almost surely, for a.a. $s$.

\smallskip
We now prove 7). We have
\begin{align*}
[\ol{X^i}\,,\,X^j]_t&=\sum_{k=1}^N \int_0^t H^{ik}_s\, d[x^k\,,\,X^j]_s\\
&=\sum_{k=1}^N\int_0^t H^{ik}_s\, \L^{kj}_s\,ds+\sum_{k,l=1}^N \int_0^t H^{ik}_s\, S^{kj}_l(s)\, dX^l_s\,,
\end{align*}
but also
$$
[\ol{X^i}\,,\,X^j]_t=\d_{ij}\,t+\sum_{k=1}^N\int_0^t \ol{S^{ik}_j(s)}\, dX^k_s\,.
$$
Again, by the uniqueness lemma we get
$$
\d_{ij}=\sum_{k=1}^NH^{ik}_s\, \L^{kj}_s=\sum_{k=1}^N \ol{\L^{ik}_s}\, \L^{kj}_s\,.
$$
This proves the announced unitarity.
\end{proof}

\medskip
We can now state one of our main result concerning the structure of complex normal martingales in $\CC^N$.

\begin{theorem}\label{T:CtoRcont}
Let $X$ be a normal martingale in $\CC^N$ satisfying the structure equations
\begin{align*}
[X^i\,,\,X^j]_t&=\int_0^t \L^{ij}_s\, ds+\sum_{k=1}^N \int_0^t S^{ij}_k(s)\, dX^k_s\\
[\ol{X^i}\,,\, X^j]_t&=\d_{ij}\, t+\sum_{k=1}^N \int_0^t \ol{S^{ik}_j(s)}\, dX^k_s\,.
\end{align*}
Then the matrix $\L$ admits a decomposition of the form
$$
\L=V\,V^t
$$
for some unitary $V$ of $\CC^N$. 
For any such unitary $V$, put $R=V^*\circ S$. Then $R$ is a real doubly-symmetric 3-tensor.
\end{theorem}
\begin{proof}
This is exactly the same proof as for Theorem \ref{T:UUt} : the decomposition of $\L$ comes from Takagi's Theorem, the expression of $R^{ij}_k$ in terms of the coefficients of $V$ and of the $S^{ij}_k$'s is transformed with the help of the relation XX. One then see that $R$ satisfies the symmetry property which makes it real.
\end{proof}

\subsection{Complex Unitary Transforms of Real Normal Martingales}

From the result above concerning real normal martingales, we shall deduce easily the corresponding behavior of complex normal martingales, as they are obtained by unitary transforms of real normal martingales. 

In the following we shall be interested in the following objects. 
Let $Y$ be a normal martingale on $\RR^N$, satisfying a structure equation with constant 3-tensor $T$. Let $U$ be a unitary operator on $\CC^N$. Injecting canonically $\RR^N$ into $\CC^N$, we consider the complex martingale $X_t=UY_t$, $t\in\Rp$. We consider the 3-tensor
$$
S=U\circ T\,.
$$
We also put
$$
\L=U\,U^t\,.
$$
We choose the following notations. Let $\rW=\{w_1,\ldots, w_k\}\subset \RR^N$ be the orthogonal system associated to $T$, that is, the directions of non-vanishing eigenvalues of $T$. Let $\rW^\perp$ be its orthogonal space in $\RR^N$, that is the null space of $T$, and let us choose an orthonormal basis $\{\wh{w}_1,\ldots, \wh{w}_{N-k}\}$ of $\rW^\perp$ . Let $\rV=U\rW=\{Uw_1,\ldots,Uw_k\}\subset\CC^N$, this set coincides with the orthogonal system associated to $S$ that is, the directions of non-vanishing eigenvalues of $S$. We consider the set $\wh{\rV}=\{\wh{v}_1,\ldots,\wh{v}_{N-k}\}$, where $\wh{v_i}=U\wh{w_i}$, for all $i=1,\ldots, N-k$. We denote by $\rV_\RR$ and $\wh{\rV}_\RR$ the following real subspaces of $\CC^N$ seen as a $2N$-dimensional real vector space:
\begin{align*}
\rV_\RR&=\RR\, v_1\oplus\ldots\oplus\RR\, v_k\,\\
\wh{\rV}_\RR&=\RR\wh{v}_1\oplus\ldots\oplus\RR\wh{v}_{N-k}\,.
\end{align*}
In particular note that $\rV_\RR\oplus\wh{\rV}_\RR=U\,\RR^N$ is a $N$-dimensional real subspace of $\CC^N$.

Finally we denote by $\Pi_S$ the orthogonal projector from $\CC^N$ onto $\wh{\rV}_\RR$ where both spaces are seen as real vector spaces.

\begin{theorem}\label{T:normalC}
With the notations above, the complex martingale $X$ satisfies the following two ``structure equations''
\begin{align}
[X^i\,,X^j]_t&=\L_{ij}\, t+\sum_{k=1}^N S^{ij}_k\, X^k_t\,,\label{E:ESX}\\
[\ol{X^i}\,,X^j]_t&=\d_{ij}\, t+\sum_{k=1}^N \ol{S^{ik}_l}\, X^k_t\,.\label{E:ESX2}
\end{align}
The solutions to both Equation (\ref{E:ESX}) and Equation (\ref{E:ESX2}) are unique in distribution. This distribution is described as follows. The process $X$ is valued in $U\,\RR^N=\rV_\RR\oplus\wh{\rV}_\RR$. The  continuous part of $X$ is
$$
X^c_t=\int_0^t \Pi_S(dX_s)\,,
$$
which lives in $\wh{\rV}_\RR$.
The jumps of $X$ only happen at totally inaccessible times and they satisfy, almost surely
$$
\D X_t(\o)\in\rV_\RR\,.
$$ 
In other words, let $W$ be a $N-k$ dimensional Brownian motion with values in the real space $\wh{\rV}_\RR$. For every $v\in\rV$, let $N^v$ be a Poisson process with intensity $\norme{v}^{-2}$. We suppose $W$ and all the $N^v$ to be independent processes. Then the martingale
\begin{equation}\label{E:explicit}
X_t=W_t+\sum_{v\in\rV} \left(N_t^v-\frac{1}{\normca{v}}\, t\right)\, v
\end{equation}
satisfies the structure equation (\ref{E:ESX}).

The process $X$ possesses the chaotic representation property.
\end{theorem}
\begin{proof}
The martingale $Y$ satisfies the structure equation
$$
[Y^i\,,\,Y^j]_t=\delta_{ij}\, t+\sum_{k=1}^N T^{ij}_k\, Y^k_t\,.
$$
The process $X_t=U\, Y_t$, $t\in\Rp$, is a martingale with values in $\CC^N$. Clearly we have
\begin{align*}
[X^i\,,\,X^j]_t&=\sum_{m,n=1}^N u_{im}\,u_{jn}\, [Y^m\,,\,Y^n]_t\\
&=\sum_{m,n=1}^N u_{im}\,u_{jn} \,\delta_{mn} t+\sum_{m,n=1}^N\sum_{p=1}^N u_{im}\,u_{jn}\, T^{mn}_p\, Y^p_t\\
&=\sum_{m=1}^N u_{im}\,u_{jm} \,t+\sum_{m,n=1}^N\sum_{p=1}^N\sum_{k=1}^N u_{im}\,u_{jn}\, \ol{u_{kp}}\, T^{mn}_p\, X^k_t\\
&=S^{ij}_0 \,t+\sum_{k=1}^N S^{ij}_k\, X^k_t\,.
\end{align*}
This gives (\ref{E:ESX}).

With a similar computation, we get
\begin{align*}
[\ol{X^i}\,,\,X^j]_t&=\sum_{m,n=1}^N \ol{u_{im}}\,u_{jn}\, [Y^m\,,\,Y^n]_t\\
&=\sum_{m,n=1}^N \ol{u_{im}}\,u_{jn} \,\delta_{mn} t+\sum_{m,n=1}^N\sum_{p=1}^N \ol{u_{im}}\,u_{jn}\, T^{mn}_p\, Y^p_t\\
&=\sum_{m=1}^N \ol{u_{im}}\,u_{jm} \,t+\sum_{m,n=1}^N\sum_{p=1}^N\sum_{k=1}^N \ol{u_{im}}\,u_{jn}\, \ol{u_{kp}}\, T^{mn}_p\, X^k_t\\
&=\sum_{m=1}^N \delta_{ij}\,t+\sum_{m,n=1}^N\sum_{p=1}^N\sum_{k=1}^N \ol{u_{im}\,u_{kp}\, \ol{u_{jn}}}\, T^{mp}_n\, X^k_t\\
&=\delta_{ij} \,t+\sum_{k=1}^N \ol{S^{ik}_j}\, X^k_t\,.
\end{align*}
This gives (\ref{E:ESX2}).

Let us now prove uniqueness in law for the solutions of (\ref{E:ESX}) and (\ref{E:ESX2}). Let $Z$ be another solution of the two equations. Consider the unitary operator $U$ associated to the 3-tensor $S$, that is for which the 3-tensor $T=U^*\circ S$ is real and doubly-symmetric. Put $A_t=U^*_tZ_t$ for all $t\in\RR^p$. Then we get 
\begin{align*}
[A^i\,,\,A^j]_t&=\sum_{m,n=1}^N \ol{u_{mi}}\,\ol{u_{nj}}\, [Z^m\,,\,Z^n]_t\\
&=\sum_{m,n=1}^N \ol{u_{mi}}\,\ol{u_{nj}}\,S^{mn}_0 t+\sum_{m,n=1}^N\sum_{p=1}^N \ol{u_{mi}}\,\ol{u_{nj}}\, S^{mn}_p\, Z^p_t\\
&=\sum_{m,n=1}^N \sum_{p=1}^n\ol{u_{mi}}\,\ol{u_{nj}}\, u_{mp}\,u_{np} \,t+\sum_{m,n=1}^N\sum_{p=1}^N\sum_{k=1}^N \ol{u_{mi}}\,\ol{u_{nj}}\, u_{pk}\, S^{mn}_p\, A^k_t\\
&=\delta_{ij} \,t+\sum_{k=1}^N R^{ij}_k\, A^k_t
\end{align*}
and
\begin{align*}
[\ol{A^i}\,,\,A^j]_t&=\sum_{m,n=1}^N {u_{mi}}\,\ol{u_{nj}}\, [\ol{Z^m}\,,\,Z^n]_t\\
&=\sum_{m,n=1}^N {u_{mi}}\,\ol{u_{nj}}\,\delta_{mn} t+\sum_{m,n=1}^N\sum_{p=1}^N {u_{mi}}\,\ol{u_{nj}}\, \ol{S^{mp}_n}\, Z^p_t\\
&=\sum_{m=1}^N {u_{mi}}\,\ol{u_{mj}} \,t+\sum_{m,n=1}^N\sum_{p=1}^N\sum_{k=1}^N \ol{\ol{u_{mi}}\,{u_{nj}}\, \ol{u_{pk}}\, S^{mp}_n}\, A^k_t\\
&=\delta_{ij} \,t+\sum_{k=1}^N \ol{R^{ik}_j}\, A^k_t\,.
\end{align*}
But as $R$ is a real-valued 3-tensor, symmetric in $i,j,k$, the last expression gives
$$
[\ol{A^i}\,,\,A^j]_t=\delta_{ij} \,t+\sum_{k=1}^N {R^{ij}_k}\, A^k_t\,.
$$
Decomposing each $A^j_t$ as $B^j_t +i C^j_t$ (real and imaginary parts), the last two relations ought to
\begin{align*}
[B^i\,,\,B^j]_t&=\d_{ij}\,t+\sum_{k=1}^N R^{ij}_k\, B^k_t\,,\\
[C^i\,,\,C^j]_t&=0\,,\\
[B^i\,,\,C^j]_t&=\sum_{k=1}^N R^{ij}_k \, C^k_t\,.
\end{align*}
This clearly implies that the $C^j$'s vanish, the processes $A^j$ are real-valued. They satisfy the same structure equation as the real process $Y$ underlying the definition of $X$. By Theorem~\ref{T:AE2} the processes $A$ and $Y$ have same law. Thus so do the processes $Z=UA$ and $X=UY$. This proves the uniqueness in law for the processes in $\CC^N$ satisfying the two equations (\ref{E:ESX}) and (\ref{E:ESX2}).

\smallskip
The projector $\Pi_S$ onto $\wh{\rV}_\RR$ is given by $\Pi_S=U\,\Pi_T\, U^*$, as can be checked easily, even though we are here considering the vector spaces as real ones. The continuous part of $X$ and the jumps of $X$ are clearly the ones of $Y$ but mapped by $U$. Hence, altogether we get
$$
X^c_t=UY^c_t=\int_0^t U\,\Pi_T(dY_s)=\int_0^t U\,U^*\Pi_S\,U(U^* dX_s)=\int_0^t \Pi_S(dX_s)\,.
$$

The part (\ref{E:explicit}) of the theorem is obvious, again by application of the map $U$. 

Finally, let us prove the chaotic representation property. The chaotic representation property for $Y$ says that every random variable $F\in L^2(\O,\rF,\PP)$, the canonical space of $Y$, can be decomposed as
$$
F=\EE[F]+\sum_{n=1}^\infty\sum_{i_1,\ldots, i_n=1}^N\int_{0\leq t_1<\ldots<t_n} f_{i_1,\ldots,i_n}(t_1,\ldots,t_n)\, dY^{i_1}_{t_1}\,\ldots\, dY^{i_n}_{t_n}\,,
$$
for some deterministic functions $f_{i_1,\ldots,i_n}$'s. But decomposing each $Y^j_t$ as $\sum_{m=1}^N \ol{u_{mj}}\, X^m_t$ shows clearly that $F$ can also be decomposed as 
$$
F=\EE[F]+\sum_{n=1}^\infty\sum_{i_1,\ldots, i_n=1}^N\int_{0\leq t_1<\ldots<t_n} g_{i_1,\ldots,i_n}(t_1,\ldots,t_n)\, dX^{i_1}_{t_1}\,\ldots\, dX^{i_n}_{t_n}\,,
$$
where the $g_{i_1,\ldots,i_n}$'s are linear combinations of the $f_{i_1,\ldots,i_n}$'s. This proves the chaotic representation property for $X$ and the theorem is completely proved.
\end{proof}

\section{Continuous-Time Limit of Complex Obtuse Random Walks}

We are now ready to consider the convergence theorem for complexe obtuse random walks.

\subsection{Convergence of the Tensors}

We are now given a time parameter $h>0$ which is meant to tend to 0 later on. This time parameter is the time step of the obtuse random walk we want to study, but note that $h$ may also appear in the internal parameters of the walk, that is, in the probabilities $p_i$ and the values $v_i$ of $X$. 

Hence, we are given  an obtuse random variable $X(h)$ in $\CC^N$, with coordinates $X^i(h)$, $i=1,\ldots,N$ and together with the random variable $X^0=\indic$. The associated 3-tensor of $X(h)$ is given by
$$
{X^i(h)}\,X^j(h)=\sum_{k=0}^N S^{ij}_k(h)\, X^k(h)\,.
$$
Considering the random walk associated to $X(h)$, that is, consider a sequence $(X_n(h))_{n\in\NN^*}$ of i.i.d. random variables in $\CC^N$ all having the same distribution as $X(h)$, the random walk is the stochastic process with time step $h$:
$$
Z^h_{nh}=\sum_{i=1}^n \sqrt{h}\, X_i(h)\,.
$$
This calls for defining
$$
\wh{X}^0=h\,X^0\qq\mbox{and}\qq\wh{X}^j(h)=\sqrt h\,X^j(h)
$$
for all $j=1,\ldots, n$. Putting $\e_0=1$ and $\e_i=1/2$ for all $i=1,\ldots, n$, we then have, for all $i,j=0,\ldots, N$
$$
\wh{X}^i(h)\,\wh{X}^j(h)=\sum_{k=0}^N\wh{S}^{ij}_k(h)\, \wh{X}^k(h)\,,
$$
where
$$
\wh{S}^{ij}_k(h)=h^{\e_i+\e_j-\e_k}\, S^{ij}_k\,.
$$
Finally, we put
$$
M^{ij}_k=\lim_{h\rightarrow0} \wh{S}^{ij}_k(h)\,,
$$
if it exists.

\begin{lemma}\label{limits}
We then get, for all $i,j,k=1,\ldots N$
\begin{align*}
M^{00}_0&=0\,,\\
M^{00}_k&=0\,,\\
M^{0k}_0&=0\,,\\
M^{k0}_0&=0\,,\\
M^{ij}_0&=\lim_{h\to 0}S^{ij}_0(h)\ \ \mbox{(if it exists)}\,,\\
M^{i0}_k&=0\,,\\
M^{0j}_k&=0\,,\\
M^{ij}_k&=\lim_{h\to 0}h^{1/2}\, S^{ij}_k(h)\ \ \mbox{(if it exists)}\,.
\end{align*}
\end{lemma}
\begin{proof}
These are direct applications of the definitions and the symmetries verified by the $S^{ij}_k(h)$'s. For example:
$$
\wh{S}^{0j}_k(h)=h\, S^{0j}_k(h)=h\, S^{j0}_k=h\,\d_{jk}\,.
$$
This gives immediately that $M^{0j}_k=0$.
And so on for all the other cases. 
\end{proof}

\begin{proposition}\label{Ltensor}
Under the hypothesis that the limits above all exist, the 3-tensor $M$, restricted to its coordinates  $i,j,k=1,\ldots, N$, is a doubly-symmetric 3-tensor of $\CC^N$. 
\end{proposition}
\begin{proof}
Let us check that $(M^{ij}_k)_{i,j,k=1,\ldots,N}$ satifies the three conditions for being a doubly-symmetric 3-tensor. Recall that for these indices, we have
$$
M^{ij}_k=\lim_{h\to0}h^{1/2}\, S^{ij}_k(h)\,.
$$
The first condition $M^{ij}_k=M^{kj}_i$ is obvious from the same property of $S^{ij}_k(h)$ and passing to the limit.

We wish now to prove that 
$\sum_{m=1}^N M^{im}_j\,{M^{kl}_m}$ is symmetric in $(i,k)$. The corresponding property for $S(h)$ gives
$$
S^{i0}_j(h)\,{S^{kl}_0}(h)+\sum_{m=1}^N S^{im}_j(h)\,{S^{kl}_m}(h)=
S^{k0}_j(h)\,{S^{il}_0}(h)+\sum_{m=1}^N S^{km}_j(h)\,{S^{il}_m}(h)\,.
$$
In particular, multiplying by $h$, we get
$$
h\,\d_{ij}\,S^{kl}_0(h)+\sum_{m=1}^N \wh{S}^{im}_j(h)\,{\wh{S}^{kl}_m(h)}=h\,\d_{kj}\,S^{il}_0(h)+\sum_{m=1}^N \wh{S}^{km}_j(h)\,\wh{S}^{il}_m(h)\,.
$$
By hypothesis $\lim_{h\to0} S^{kl}_0(h)$ and $\lim_{h\to0}S^{il}_0(h)$ exist hence, passing to the limit, we get
$$
\sum_{m=1}^N M^{im}_j\,{M^{kl}_m}=\sum_{m=1}^N M^{km}_j\,M^{il}_m\,,
$$
which is the second symmetry asked to $M$ for being doubly-symmetric.

The third symmetry is obtained in a similar way. Indeed, we have
$$
S^{i0}_j(h)\,\ol{S^{l0}_k}(h)+\sum_{m=1}^N S^{im}_j(h)\,\ol{S^{lm}_k}(h)=
S^{k0}_j(h)\,\ol{S^{l0}_i}(h)+\sum_{m=1}^N S^{km}_j(h)\,\ol{S^{lm}_i}(h)\,.
$$
This gives, multiplying by $h$ again
$$
h\,\d_{ij}\,\d_{lk}+\sum_{m=1}^N \wh{S}^{im}_j(h)\,\ol{\wh{S}^{lm}_k}(h)=
h\,\d_{kj}\,\\d_{li}+\sum_{m=1}^N \wh{S}^{km}_j(h)\,\ol{\wh{S}^{lm}_i}(h)\,.
$$
Now, passing to the limit as $h$ tends to 0, we get
$$
\sum_{m=1}^N M^{im}_j\,\ol{M^{lm}_k}=
\sum_{m=1}^N M^{km}_j\,\ol{M^{lm}_i}\,.
$$
This gives the last required symmetry.
\end{proof}

\subsection{Convergence in Distribution}

We can now give our convergence in distribution theorem.

\begin{theorem}
Let $X(h)$ be an obtuse random variable on $\CC^N$, depending on a parameter $h>0$, let $S(h)$ be its associated doubly symmetric 3-tensor. Let $(X_{n}(h))_{n\in\NN^*}$ be a sequence of i.i.d. random variables with same law as $X(h)$. Consider the discrete-time random walk
$$
Z^h_{nh}=\sum_{i=1}^n \sqrt{h}\, X_i(h)\,.
$$
If the limits 
$$
M^{ij}_0=\lim_{h\to0} S^{ij}_0(h)
$$
and
$$
M^{ij}_k=\lim_{h\to0}\sqrt h\, S^{ij}_k(h)
$$ 
exist for all $i,j,k=1,\ldots, N$, then the process $Z^h$ converges in distribution to the normal martingale $Z$ in $\CC^N$ solution of the structure equations
\begin{align}
[Z^i\,,Z^j]_t&=M^{ij}_0\, t+\sum_{k=1}^N M^{ij}_k\, Z^k_t\,,\label{E:ESfin}\\
[\ol{Z^i}\,,Z^j]_t&=\d_{ij}\, t+\sum_{k=1}^N \ol{M^{ik}_j}\, Z^k_t\,.\label{E:ESfin2}
\end{align}
\end{theorem}
\begin{proof}
For each $h>0$, the random variables $X(h)$ can be written $U(h)\,Y(h)$ for a unitary operator $U(h)$ on $\CC^N$ and a real obtuse random variable $Y(h)$ in $\RR^N$. In terms of the associated 3-tensors, recall that this means
$$
S(h)=U(h)\circ T(h)
$$
or else
$$
T(h)=U(h)^*\circ S(h)\,.
$$
Consider any sequence $(h_n)_{n\in\NN}$ which tends to 0. By hypothesis $S(h_n)$ converges to $M$. Furthermore, as the sequence $(U(h_n))_{n\in\NN}$ lives in the compact group $\rU(\CC^N)$ it admits a subsequence $(h_{n_k})_{k\in\NN}$ converging to some unitary $V$. As a consequence the sequence $(T(h_{n_k}))_{k\in\NN}$ converges to a real 3-tensor $N=V\circ M$. 

The convergence of the 3-tensors $(T(h_{n_k}))_{k\in\NN}$ to $N$ imply the convergence in distribution of the associated real martingales $Y_{h_{n_{k_i}}}$, for a subsequence $(h_{n_{k_i}})_{i\in\NN}$, by G. Taviot's Thesis (cf \cite{Tav}, Proposition 4.2.3, Proposition 4.3.2., Proposition 4.3.3).  The limit is a real normal martingale $Y$ whose associated 3-tensor is $N$. 

Applying the unitary operators $U_{h_{n_{k_i}}}$, which converge to $V$, we have the convergence in law of the process $Z^{h_{n_{k_i}}}$ to the process $Z=VY$.  By Theorem \ref{T:normalC} the process $Z$ is solution of the complex structure equations associated to the tensor $S$. 

For the moment we have proved that for every sequence $(h_n)_{n\in\NN}$ there exists a subsequence $(h_{n_j})_{j\in\NN}$ such that $Z^{h_{n_j}}$ converges in law to $Z$ solution of the complex structure equations associated to $S$. As we have proved that the solutions to these equation are unique in law, the limit in law is unique. Hence the convergence is true not only for subsequences, but more generally for $h$ tending to 0. 
The convergence in law is proved.
\end{proof}

\subsection{Convergence of the Multiplication Operators}

Let us first recall very shortly the main elements of the construction and approximation developed in \cite{Att}, which will now serve us in order to prove the convergence of the multiplication operators. This convergence of multiplication operators is not so usual in a probabilistic framework, but it is the one interesting in the framework of applications in Quantum Statistical Mechanics, for it shows the convergence of the quantum dynamics of repeated interactions towards a classical Langevin equation, when the unitary interaction is unitary (cf \cite{ADP}).

\smallskip
In Subsection \ref{SS:mult} we have seen the canonical isomorphism of the canonical space $L^2(\O,\rF,\PP_S)$ of any obtuse random variable $X$, with the space $\CC^{N+1}$. Recall the basic operators $a^i_j$ that were defined there. 

When dealing with i.i.d. sequences $(X_n)_{n\in\NN}$ of copies of $X$, the canonical space is then isomorphic to the countable tensor product
$$
T\Phi=\bigotimes_{n\in\NN} \CC^{N+1}\,.
$$
When dealing with the associated random walk with time step $h$
$$
Z^h_{nh}=\sum_{i=1}^n \sqrt{h}\, X_{ih}
$$
the canonical space is naturally isomorphic to
$$
T\Phi(h)=\bigotimes_{n\in h\NN}\CC^{N+1}\,.
$$
There, natural ampliations of the basic operators $a^i_j$ are defined: the operators $a^i_j(nh)$ is the acting as $a^i_j$ of the copy $nh$ of $\CC^{N+1}$ and as the identity of the other copies.

On the other hand, when given a normal martingale $A$ in $\RR^N$ with the chaotic representation property, its canonical space $L^2(\O',\rF',\PP')$ is well-known to be naturally isomorphic to the symmetric Fock space 
$$
\Phi=\Gamma_s(L^2(\Rp\,;\, \CC^N))\,,
$$ 
via a unitary isomorphism denoted by $U_A$.
This space is the natural space for the quantum noises $a^i_j(t)$, made of the time operator $a^0_0(t)=tI$, the creation noises $a^0_i(t)$, the annihilation noises $a^i_0(t)$ and the exchange processes $a^i_j(t)$, with $i,j=1,\ldots N$ (cf \cite{Mey}).

\smallskip
The main constructions and results developed in \cite{Att} are the following:

\smallskip\noindent
-- each of the spaces $T\Phi(h)$ can be naturally seen as concrete subspace of $\Phi$;

\smallskip\noindent
-- when $h$ tends to 0 the subspace $T\Phi(h)$ fills in the whole space $\Phi$, that is, concretely, the orthogonal projector $P_h$ onto $T\Phi(h)$ converges strongly to the identity $I$;

\smallskip\noindent
-- the basic operators $a^i_j(nh)$, now concretely acting on $\Phi$, converge to the quantum noises, that is, more concretely the operator
$$
\sum_{n\,;\ nh\leq t} h^{\e^i_j} a^i_j(nh)
$$ 
converges strongly to $a^i_j(t)$ on a certain domain $\rD$ (which we shall not make explicit here, please cf \cite{Att}), where 
$$
e^i_j=\begin{cases} 1&\mbox{if } i=j=0\,,\\
1/2&\mbox{if } i=0, j\not =0\ \mbox{or } i\not=0, j=0\,,\\
0&\mbox{if }i,j\not=0\,.
\end{cases}
$$

Finally recall the representation of the multiplication operators for real-valued normal martingales in $\RR^N$ (cf \cite{A-P2}).

\begin{theorem}\label{T:mult_cont}
If $A$ is a normal martingale in $\RR^N$ with the chaotic representation property and satisfying the structure equation
$$
[A^i\,,\,A^j]_t=\d_{ij}\,t+\sum_{k=1}^N N^{ij}_k\, A^k_t\,,
$$
then its multiplication operator acting of $\Phi=\Gamma_s(L^2(\Rp\,,\,\CC^N))$ is equal to 
$$
U_A\, \rM_{A^i_t}\, U_A^*=a^0_i(t) +a^0_i(t)+\sum_{j,k=1}^N N^{ij}_k\, a^j_k(t)
$$
or else
$$
U_A\,rM_{A^i_t}\,U_A^*=\sum_{j,k=0}^N N^{ij}_k\, a^j_k(t)
$$
if one extends the coefficients $N^{ij}_k$ to the 0 index, by putting $N^{ij}_0=\d_{ij}$.
\end{theorem}
Once this is recalled, the rest is now rather easy. We can prove the convergence theorem for the multiplication operators.

\begin{theorem}
The operators of multiplication $\rM_{Z^h_t}$, acting of $\Phi$, converge strongly on $\rD$ to the operators 
\begin{equation}\label{E:rY}
\rZ_t=\sum_{j,k=0}^n M^{ij}_k\, a^j_k(t)\,.
\end{equation}
These operators are the operators of multiplication by $Z$ the complex martingale satisfying
\begin{equation}\label{E:ESY}
[Z^i\,,Z^j]_t=M^{ij}_0\, t+\sum_{k=1}^N M^{ij}_k\, Z^k_t
\end{equation}
and
\begin{equation}\label{E:ESY2}
[\ol{Z^i}\,,Z^j]_t=\d_{ij}\, t+\sum_{k=1}^N \ol{M^{ik}_l}\, Z^k_t\,.
\end{equation}\end{theorem}
\begin{proof}
The convergence toward the operator $\rZ_t$ given by (\ref{E:rY}) is a simple application of the convergence theorems of \cite{Att}, let us detail the different cases.

If $j,k\not=0$, we know that $\sqrt h\, S^{ij}_k$ converges to $M^{ij}_k$ and by \cite{Att} we have that 
$\sum_{m=1}^{\left[t/h\right]} a^j_k(m)$ converges to $a^j_k(t)$.

If $j=0$ and $k\not=0$, we know that $S^{ij}_0$ converges to $M^{ij}_0$ and that 
$\sum_{m=1}^{\left[t/h\right]} \sqrt h\, a^j_0(m)$ converges to $a^j_0(t)$.

If $k=0$ and $j\not =0$, we know that $S^{i0}_k$ converges to $M^{i0}_k$ (actually their are all equal to $\d_{ik}$) and that 
$\sum_{m=1}^{\left[t/h\right]} \sqrt h\, a^0_k(m)$ converges to $a^0_k(t)$.

\smallskip
The fact that $\rZ_t$ is indeed the multiplication operator by the announced normal martingale comes as follows. The martingale $Z$ is the image $UA$, under a unitary operator $U$ of some real normal martingale $A$. The 3-tensor $M$ is the image $U\circ N$, under the unitary operator $U$, of some real tensor $N$. The real normal martingale $A$ associated to the real 3-tensor $N$ has its multiplication operator equal to
$$
U_A\, \rM_{A^i_t}\, U_A^*=\sum_{j,k=0}^N N^{ij}_k\, a^j_k(t)
$$
by Theorem \ref{T:mult_cont}. As $Z_t$ is equal to $UA_t$ its canonical space is the same as the one of $A$, only the canonical isomorphism is modified by a change of basis. The rest of the proof is then exactly similar to the one of Proposition \ref{P:change_opmult}.
\end{proof}

\section{Examples}

We shall detail 2 examples in dimension 2, showing up typical different behaviors.

\bigskip
 The first one is the one we have followed along this article, let us recall it. We are given an obtuse random variable $X$ in $\CC^2$ taking the values
$$
v_1=\left(\begin{matrix}i\\1\end{matrix}\right)\ ,\qq 
v_2=\left(\begin{matrix}1\\-1+i\end{matrix}\right)\ ,\qq 
v_3=-\frac15\left(\begin{matrix}3+4i\\1+3i\end{matrix}\right)
$$
with probabilities $p_1=1/3$, $p_2=1/4$ and $p_3=5/{12}$ respectively. Then the 3-tensor $S$ associated to $X$ is given by 
\begin{align*}
S^0&=\left(\begin{matrix}1&0&0\\\ecarte 0&1&0\\\ecarte0&0&1\end{matrix}\right)\,,
\qq S^1=\left(\begin{matrix}0&1&0\\\ecarte-\frac15(1-2i)&0&-\frac25(2+i)\\\ecarte
-\frac25(1-2i)&0&\frac15(2+i)\end{matrix}\right)\\
S^2&=\left(\begin{matrix}0&0&1\\\ecarte-\frac25(1-2i)&0&\frac15(2+i)\\\ecarte
\frac15(1-2i)&-i&-\frac15(1-2i)\end{matrix}\right)\,.
\end{align*}

Now, considering the random walk 
$$
Z^h_{nh}=\sum_{i=1}^n \sqrt{h}\, X_n\,,
$$
where $(X_n)_{n\in\NN}$ is a sequence of i.i.d. random variables with same law as $X$, the continuous-time limit of $Z^h$ is the normal martingale in $\CC^2$ with associated tensor given by the limits of Lemma \label{limits}. Here we obtain, for all $i,j,k=1,2$
$$
M^{ij}_k=0
$$
and 
$$
M_0=\left(\begin{matrix}-\frac15(1-2i)&-\frac25(1-2i)\\\\-\frac25(1-2i)&\frac15(1-2i)\end{matrix}\right)\,.
$$
The limit process $(Z_t)_{t\in\Rp}$ is a normal martingale in $\CC^2$, solution of the structure equations
\begin{align*}
[Z^i\,,\,Z^j]_t=M^{ij}_0\, t\\
[\ol{Z^i}\,,\,Z^j]_t=\d_{ij}\, t\,.
\end{align*}
It is then rather easy to find a unitary matrix $V$ such that $V\,V^t=M_0$, we find
$$
V=\left(\begin{matrix}\frac{2+i}{\sqrt{10}}&\frac{i}{\sqrt{2}}\\\\
\frac{-1+2i}{\sqrt{10}}&\frac{1}{\sqrt{2}}\end{matrix}\right)\,,
$$
for example. Following our results on complex normal martingales, this means that the process $Z$ has the following distribution: given a 2-dimensional real Brownian motion $W=(W^1\,,\,W^2)$ then 
$$
\begin{cases}
Z^1_t&=\frac{2+i}{\sqrt{10}}\, W^1_t+\frac{i}{\sqrt{2}}\, W^2_t\\\\
Z^2_t&=\frac{-1+2i}{\sqrt{10}}\, W^1_t+\frac{1}{\sqrt{2}}\, W^2_t\,.
\end{cases}
$$

\bigskip
For the second example, we consider a fixed parameter $h>0$. We consider the obtuse random variable $X(h)$ in $\CC^2$ whose values are
$$
v_1=\frac{1}{\sqrt 2}\left(\begin{matrix} i\\\ecarte 1\end{matrix}\right)\ ,\qq 
v_2=\frac{1}{\sqrt{2h}}\left(\begin{matrix}1-i\sqrt{h}\\\ecarte{i-\sqrt{h}}\end{matrix}\right)\ ,\qq 
v_3=-\frac{1}{\sqrt{2}}\left(\begin{matrix}{2\sqrt{h}+i}\\\ecarte{1+2i\sqrt h}\end{matrix}\right)
$$
with probabilities $p_1=1/2$, $p_2=h/(1+2h)$ and $p_3=1/(2+4h)$ respectively. Then the 3-tensor $S$ associated to $X$ is given by the following, where we have only detailed the leading orders in $h$
\begin{align*}
S^0&=\left(\begin{matrix}1&0&0\\\ecarte 0&1&0\\\ecarte0&0&1\end{matrix}\right)\,,
\qq S^1=\left(\begin{matrix}0&1&0\\\ecarte0&\frac{1}{2\sqrt{2h}}+O(1)&\frac{-i}{2\sqrt{2h}}+O(1)\\\ecarte
i&\frac{i}{2\sqrt{2h}}+O(1)&\frac{1}{2\sqrt{2h}}+O(1)\end{matrix}\right)\\
S^2&=\left(\begin{matrix}0&1&0\\\ecarte i&\frac{i}{2\sqrt{2h}}+O(1)&\frac{1}{2\sqrt{2h}}+O(1)\\\ecarte
0&\frac{-1}{2\sqrt{2h}}+O(1)&\frac{i}{2\sqrt{2h}}+O(1)\end{matrix}\right)\,.
\end{align*}

The renormalized 3-tensor converges to the 3-tensor
\begin{align*}
M^1&=\frac{1}{2\sqrt2}\,\left(\begin{matrix} 1&-i\\\ecarte i&1\end{matrix}\right)\\
M^2&=\frac{1}{2\sqrt2}\,\left(\begin{matrix} i&1\\\ecarte -1&i\end{matrix}\right)
\end{align*}
and the matrix
$$
M^0=\left(\begin{matrix} 0&i\\\ecarte i&0\end{matrix}\right)\,.
$$
In order to diagonalize the 3-tensor, we solve
$$
(M^{ij}_1x+M^{ij}_2y)_{i,j=1,2}=\left(\begin{matrix}x\\\ecarte y\end{matrix}\right)\otimes\left(\begin{matrix}x\\\ecarte y\end{matrix}\right)=\left(\begin{matrix}x^2&xy\\\ecarte xy&y^2\end{matrix}\right)\,.
$$
There is a unique solution 
$$
v=\frac{\sqrt2}2\, \left(\begin{matrix}1\\\ecarte i\end{matrix}\right)\,.
$$
This means that the continuous-time limit process $Z$ is a compensated Poisson process in the direction $v$. 

This is all for the information which is given by the 3-tensor. If we want to know the direction where the process is Brownian, we need to look at the decomposition of $M^0$ as $V\, V^t$ for a unitary $V$. We easily find
$$
V=\frac 12\,\left(\begin{matrix}1+i&1-i\\\ecarte 1+i&1-i\end{matrix}\right)\,.
$$
This unitary operator is the one from which has been rotated a real Brownian motion in order to land in the orthogonal space of $\left(\begin{matrix}i\\1\end{matrix}\right)$. That is we seek for a direction $w$ in $\RR^2$ such that $V\,w$ is proportional to $\left(\begin{matrix}1\\i\end{matrix}\right)$. We easily find $ w=\left(\begin{matrix}1\\1\end{matrix}\right)$ and the process $Z$ is a Brownian motion in the direction $\left(\begin{matrix}i\\1\end{matrix}\right)$.

The process $Z$ is finally described as follows, let $N$ and $W$ be a standard Poisson process and a Brownian motion, respectively, independant of each other. Then
$$
\begin{cases}
Z^1_t&=\frac{1}{\sqrt{2}}\, (N_t-t)+i W_t\\\\
Z^2_t&=\frac{i}{\sqrt{2}}\, (N_t-t)+\, W_t\,.
\end{cases}
$$

\bigskip
By choosing an example with two directions whose probabilities are of order $h$ and one direction's probability is of order $1-2h$, we shall end up with a 3-tensor $M$ that can be completely diagonalized and a process which is made of two compensated Poisson processes on two orthogonal directions of $\CC^2$ (cf the example at the end of \cite{A-P2} for an example in $\RR^2$).

\bigskip
\begin{multicols}{3}
\begin{minipage}{0.4\textwidth}
{\timesept St\'ephane ATTAL
\vskip -1mm
 Universit\'e de Lyon
\vskip -1mm
Universit\'e de Lyon 1, C.N.R.S.
\vskip -1mm
Institut Camille Jordan
\vskip -1mm
21 av Claude Bernard
\vskip -1mm
69622 Villeubanne cedex, France}
\end{minipage}
\columnbreak
\begin{minipage}{0.4\textwidth}
{\timesept Julien DESCHAMPS
\vskip -1mm
Dipartimento di Matematica
\vskip -1mm
Universitˆ degli Studi di Genova
\vskip -1mm
Via Dodecaneso, 35
\vskip -1mm
16146 Genova - ITALIA}
\end{minipage}
\columnbreak
\begin{minipage}{0.4\textwidth}
{\timesept Cl\'ement PELLEGRINI
\vskip -1mm
Institut de Math\'ematiques de Toulouse 
\vskip -1mm
Laboratoire de Statistique et de Probabilit\'e
\vskip -1mm
Universit\'e Paul Sabatier (Toulouse III)
\vskip -1mm
31062 Toulouse Cedex 9, France}
\end{minipage}
\end{multicols}

\end{document}